\documentclass[12pt]{amsart}

\usepackage{graphicx}
\usepackage{mathptmx}      
\usepackage{latexsym}
\usepackage{amssymb}
\usepackage{cite}
\usepackage{amsfonts}	
\usepackage{amstext}
\usepackage{amsthm,amsmath,amsfonts,amssymb,amscd,latexsym,txfonts,stmaryrd}
\usepackage{color}
\usepackage{epsfig}
\usepackage[all,cmtip]{xy}
\usepackage{enumerate}
\newtheorem{theorem}{Theorem}
\newtheorem*{theorem1*}{Theorem 1}
\newtheorem*{theorem*}{Theorem 1}
\newtheorem*{problem*}{Problem}
\newtheorem{definition}{Definition}
\newtheorem*{remarks*}{Remarks}
\newtheorem*{remark*}{Remark}
\newtheorem{lemma}{Lemma}
\newtheorem{corollary}{Corollary}
\newtheorem*{corollary1*}{Corollary 1}
\newtheorem*{corollary*}{Corollary}

\newcommand{\dsp}{\displaystyle}

\newcommand{\xymap}{\xymatrixcolsep{5pc}\xymatrixrowsep{1pc}\xymatrix}
\newcommand{\xymapp}{\xymatrixcolsep{2.5pc}\xymatrixrowsep{1pc}\xymatrix}

\newcommand{\spa}{\operatorname{span}}

\newcommand{\sra}{\shortrightarrow}

\newcommand{\ol}{\overline}

\newcommand{\R}{{\mathbb R}}

\begin{document}
\title[Lifting]{Generalized 1-Skeleta and a Lifting Result}
\author{Chris McDaniel}
\address{Dept. of Math. and Comp. Sci.\\
Endicott College\\
Beverly, MA 01915}
\email{cmcdanie@endicott.edu}






\begin{abstract}
In \cite{GZ1} Guillemin and Zara described some beautiful constructions enabling them to use Morse theory on a certain class 1-skeleta that includes all 1-skeleta of simple polytopes.  In this paper we extend some of the notions and constructions from \cite{GZ1} to a larger class of 1-skeleta that includes all 1-skeleta of projected simple polytopes.  As an application of these ideas we prove a lifting result for 1-skeleta, which yields a characterization of 1-skeleta of projected simple polytopes.
\end{abstract}

\maketitle










\section{Introduction}
\label{intro}
A 1-skeleton (in $\R^n$) is a triple $(\Gamma,\alpha,\theta)$ where $\Gamma$ is a $d$-valent graph, $\theta$ is a \emph{connection} that matches edges at adjacent vertices, and $\alpha$ is an \emph{axial function} that assigns to the edges issuing from each vertex of $\Gamma$ pairwise linearly independent vectors (in $\R^n$)  satisfying certain coplanarity conditions determined by $\theta$.  The study of 1-skeleta was motivated by work of Chang and Skjelbred \cite{CS}, and subsequent work of Goresky, Kottwitz, and MacPherson \cite{GKM}, where it was observed that the equivariant cohomology of a certain type of topological space equpped with a torus action, now known as a (-n equivariantly formal) GKM space, could be computed from the linear graph associated to its $0$- and $1$-dimensional torus orbits.  An important example of a smooth GKM space, i.e. a GKM manifold, is a smooth projective toric variety, taken either with its full torus action or with a suitable subtorus action; here the 1-skeleton is that of the associated simple polytope or, for subtorus actions, a linear projection of it.

In a series of papers \cite{GZ0,GZ1,GZ2}, Guillemin and Zara formulated the notion of a 1-skeleton as described above, and showed that many topological techniques and theorems regarding GKM manifolds have combinatorial analogues on 1-skeleta.  In particular, Guillemin and Zara described combinatorial analogues for symplectic blow-up and symplectic reduction, and developed a combinatorial analogue of Morse theory for 1-skeleta.  In \cite{GZ1} Guillemin and Zara described a class of $3$-independent 1-skeleta they called \emph{noncyclic}, on whose members their reduction operation and their Morse theory apply.  Specifically a $d$-valent $k$-independent ($k\geq 3$) noncyclic 1-skeleton has a Morse function, and has planar slices that are convex polygons.  The reduced spaces, or \emph{cross sections}, are the level sets of the Morse function, and are themselves $(d-1)$-valent $(k-1)$-independent 1-skeleta.  An important result of \cite{GZ1} proves that all cross sections of a given noncyclic 1-skeleton are related by a finite sequence of blow-ups and blow-downs.     

While the noncyclic class contains all 1-skeleta of simple polytopes, the $3$-independence condition excludes many 1-skeleta of projected simple polytopes.  In Section 3 we extend the noncyclic class to a larger class of 1-skeleta called \emph{reducible} by dropping the $3$-independence condition, and replacing ``planar slices'' with subskeleta that we call \emph{$2$-faces}.  The reducible class then contains all 1-skeleta of simple polytopes, as well as all 1-skeleta of projected simple polytopes. 
The reduction operation on the reducible class is slightly complicated; without $3$-independence, the cross sections of a reducible 1-skeleton degenerate, but they do so in a manageable way.

A \emph{compatibility system}, $\lambda$, for the triple $(\Gamma,\alpha,\theta)$ is a system of positive constants related to the compatibility of $\alpha$ and $\theta$.  A compatibility system on a 1-skeleton is uniquely determined by the pairwise independence condition on $\alpha$, hence $\lambda$ is sometimes omitted from the notation.  In Section 3 we define a \emph{generalized 1-skeleton} as a quadruple $(\Gamma,\alpha,\theta,\lambda)$ where $\Gamma$, $\theta$, and $\lambda$ are as above, and $\alpha$ is almost an axial function as defined above, except that we drop the pairwise linearly independent condition; in particular a 1-skeleton is a generalized 1-skeleton.  We show that the reduction operation of \cite{GZ1} applies almost verbatim to reducible 1-skeleta, except that the cross sections of a \emph{reducible} 1-skeleton are \emph{generalized} 1-skeleta. 

A reducible $d$-valent $d$-independent 1-skeleton (e.g. the 1-skeleton of a simple $d$-polytope) has restrictions on its cross sections that are inherited through projection.  In particular, the cross section of a general reducible 1-skeleton comes with two natural connections (\emph{up} and \emph{down}) and two natural compatibility systems (\emph{up} and \emph{down}).  It is straight forward to show that the up datum agrees with the down datum on the cross sections of a $d$-valent $d$-independent 1-skeleton.  In general, the up and down connections on each cross section of a reducible 1-skeleton agree if and only if its $2$-faces have \emph{trivial normal holonomy}, and, similarly, the up and down compatibility systems on each cross section agree if and only if the $2$-faces are \emph{level}.  What is perhaps surprising is that these two properties of $2$-faces are sufficient to detect projections of $d$-independent reducible 1-skeleta.  We say a $d$-valent 1-skeleton has a \emph{total lift} if it is the projection of a $d$-valent $d$-independent 1-skeleton.  In Section 4 we prove the main result in this paper which characterizes reducible 1-skeleta admitting a total lift.
\begin{theorem}
\label{thm:ch2mainresult}
Let $(\Gamma,\alpha,\theta,\lambda)$ be a $d$-valent reducible 1-skeleton in $\R^n$.  Then $(\Gamma,\alpha,\theta,\lambda)$ has a total lift if and only if\\
${\dsp\left(\dagger\right) \ \ \  \text{Every $2$-face of $(\Gamma,\alpha,\theta,\lambda)$ is level and has trivial normal holonomy}.}$
\end{theorem}

In Section 2 we make the important observation that a reducible $d$-valent $d$-independent 1-skeleton in $\R^d$ corresponds to a complete simplicial fan in $\left(\R^d\right)^*$.  Moreover, under this correspondence an \emph{embedding} for the 1-skeleton corresponds to a strictly convex conewise linear function on the fan.  Thus we have the following corollary to Theorem \ref{thm:ch2mainresult}, characterizing 1-skeleta of projected simple polytopes.

\begin{corollary}
\label{cor:simplepolytope}
A $d$-valent 1-skeleton $(\Gamma,\alpha,\theta,\lambda)$ is the 1-skeleton of a projected simple polytope if and only if $(\Gamma,\alpha,\theta,\lambda)$ is reducible, all its $2$-faces are level with trivial normal holonomy, and it admits an embedding $F\colon V_\Gamma\rightarrow \R^n$.
\end{corollary}  

The proof of Theorem \ref{thm:ch2mainresult} follows from four lemmata.  In Lemma \ref{lem:equivblowup}, we strengthen a result of Guillemin and Zara \cite[Theorem 2.3.2]{GZ1} that describes the change in cross sections in passing over a critical value as a blow-up/blow-down operation.  Using a \emph{cutting} technique described in \cite{GZ1}, Lemma \ref{lem:cutlift} shows that any reducible $d$-valent 1-skeleton satisfying $\left(\dagger\right)$ can be realized as the cross section of a reducible $(d+1)$-valent 1-skeleton satisfying $\left(\dagger\right)$.  Lemma \ref{lem:blowupcommuteslift} shows that total lift-ability is preserved by a blow-up/blow-down operation.  Finally, using Lemma \ref{lem:blowupcommuteslift}, Lemma \ref{lem:downcrosslift} shows that every cross section of a reducible 1-skeleton satisfying $\left(\dagger\right)$ necessarily has a total lift. 

This paper is divided into five sections.  In Section 2 we introduce some preliminary notions, closely following \cite{GZ1}.  We also discuss the relationship between complete simplicial fans and 1-skeleta.  In Section 3 we introduce the notion of a generalized 1-skeleton, and extend the necessary components of the machinery developed in \cite{GZ1}.  We also prove Lemma \ref{lem:equivblowup} and Lemma \ref{lem:cutlift}.  In Section 4 we prove Lemma \ref{lem:blowupcommuteslift} and Lemma \ref{lem:downcrosslift}.  We then prove Theorem \ref{thm:ch2mainresult} and Corollary \ref{cor:simplepolytope}.  In Section 5 we give some concluding remarks. 

This paper represents part of a thesis written at University of Massachusetts Amherst under the supervision of Professor Tom Braden, to whom I am greatly indebted for his endless encouragement, patience, and advice, and without whom I could not have completed this.

\section{1-Skeleta}
\label{sec:1}
\subsection{Preliminaries}
All graphs in this paper are assumed to be simple, i.e. no multiple edges and no single edge loops.  Let $\Gamma$ denote a connected regular graph with vertex set $V_\Gamma$ and \emph{oriented} edge set $E_\Gamma$, i.e. elements of $E_\Gamma$ are \emph{ordered} pairs of vertices.  An oriented edge $e\in E_\Gamma$ has an \emph{initial} vertex denoted $i(e)$, and a \emph{terminal} vertex denoted $t(e)$.  We write $e=[i(e)][t(e)]$ or if $i(e)$ and $t(e)$ are given explicitly as vertices $p$ and $q$, resp., then $e=pq$.  We write $\bar{e}$ to denote the oppositely oriented edge of $e$, i.e. if $e=pq$ then $\bar{e}=qp$.  The oriented edges \emph{at} vertex $p$ are those edges with initial vertex $p$, denoted $E^p$.  The \emph{valency} of $\Gamma$ is the cardinality of $E^p$, i.e. if $\left|E^p\right|=d$ we say $\Gamma$ is $d$-valent.  The first piece of structure we impose on $\Gamma$ is a perfect matching of the sets $E^{i(e)}$ and $E^{t(e)}$ for all $e\in E_\Gamma$. 

\begin{definition}
\label{def:connection}
A connection $\theta\coloneqq\{\theta_e\}_{e\in E_\Gamma}$ on $\Gamma$ is a collection of bijective maps $\theta_e\colon: E^{i(e)}\rightarrow E^{t(e)}$ indexed by the oriented edges such that $\theta_e(e)=\bar{e}$ and $\theta_e\circ\theta_{\bar{e}}=1$
We call the pair $(\Gamma,\theta)$ a ($d$-valent) graph-connection pair.
\end{definition}
Next we assign vector-valued weights to the oriented edges of $\Gamma$ in a way that coincides with a given connection $\theta$ on $\Gamma$.  See Fig.~\ref{fig:axialill}.
\begin{definition}
\label{def:axial}
A function $\alpha\colon E_\Gamma\rightarrow\R^n$ is called an \emph{axial function} for the graph connection pair $(\Gamma,\theta)$ if it satisfies the following axioms.
\begin{enumerate}[{A}1.]
\item 
\label{A.1}
 For every $p\in V_\Gamma$, the set $\{\alpha(e)\mid e\in E^p\}$ is pairwise linearly independent.
\item 
\label{A.2}
 For each $e\in E_\Gamma$, we have $\alpha(e)=-\alpha(\bar{e})$.
\item
\label{A.3}
For every $e\in E_\Gamma$ and each $e'\in E_{i(e)}\setminus\{e\}$ there exist positive constants $\lambda_e(e^\prime)$ such that
$$\alpha(e')-\lambda_e(e')\alpha(\theta_e(e'))\in\R\cdot\alpha(e).$$
\end{enumerate}
\end{definition}
The graph-connection-axial function triple is called a ($d$-valent) \emph{1-skeleton} in $\R^n$, denoted $(\Gamma,\alpha,\theta)\subset\R^n$ or simply $(\Gamma,\alpha,\theta)$.  The numbers $\lambda_e(e')$ form a \emph{compatibility system} $\lambda$ for the 1-skeleton.  Note however that $\lambda$ is completely determined by $\alpha$ and $\theta$ by the independence condition in A1; it is for this reason that $\lambda$ is omitted from the notation here.  In our subsequent generalizations, when condition A1 is dropped, it will be important to keep track of $\lambda$.  
\begin{figure}
  \includegraphics{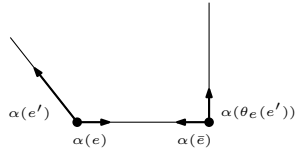}
\caption{Axial Function on  a Graph}
\label{fig:axialill}       
\end{figure}

\begin{definition}
\label{def:kind}
Given a graph $\Gamma$ we say that a function $\alpha\colon E_\Gamma\rightarrow\R^n$ is
\begin{enumerate}
\item \emph{$k$-independent} if for every vertex $p\in V_\Gamma$ and for any $k$-subset $e_1,\ldots,e_k$ of oriented edges at $p$, the set $\{\alpha(e_1),\ldots,\alpha(e_k)\}$ is linearly independent
\item \emph{effective} if the set of vectors $\alpha(E_p)\coloneqq\{\alpha(e)\mid e\in E_p\}\subset\R^n$ spans $\R^n$ for each $p\in V_\Gamma$.  
\end{enumerate}
We will say that the 1-skeleton $(\Gamma,\alpha,\theta)\subset\R^n$ is $k$-independent, resp.\ effective, if $\alpha$ is $k$-independent, resp.\ effective.  
\end{definition}
Note that every 1-skeleton is $2$-independent by definition.  Furthermore, note that if $(\Gamma,\alpha,\theta)$ is $k$-independent, then it is also $j$-independent for all $j\leq k$.  We shall see that independence conditions on a 1-skeleton severely restrict its combinatorics.  For example, the graph $\Gamma$ of a 1-skeleton $(\Gamma,\alpha,\theta)$, may, in general, admit another connection $\theta'$ distinct from $\theta$ for which $\alpha$ is compatible.  On the other hand, if $\alpha$ is $3$-independent, then $\theta$ is the unique connection on $\Gamma$ compatible with $\alpha$.

Note that a 1-skeleton is not an embedded graph \emph{\'a priori}; the axioms do not specify position vectors for the vertices of $\Gamma$.  On the other hand some embedded graphs are 1-skeleta, and conversely some 1-skeleta can be realized as embedded graphs.  The following definition makes this notion precise.
\begin{definition}
\label{def:embedding}
An \emph{embedding} of a 1 skeleton $(\Gamma,\alpha,\theta)\subset\R^n$ is a function
$$f\colon V_\Gamma\rightarrow\R^n$$
with the property that for each oriented edge $pq\in E_\Gamma$ there is a positive constant $c_{pq}\in\R_+$ such that 
$$f(q)-f(p)=c_{pq}\alpha(pq).$$
\end{definition}

If a 1-skeleton $(\Gamma,\alpha,\theta)\subset\R^n$ has an embedding $f\colon V_\Gamma\rightarrow\R^n$, then $(\Gamma,\alpha,\theta)$ can be represented as an embedded graph in the sense that $V_\Gamma$ is identified with the subset $\left\{f(p)\left|\right. p\in V_\Gamma\right\}\subset\R^n$ and the oriented edges $pq\in E_\Gamma$ are identified with the oriented straight line segments joining $f(p)$ to $f(q)$.  In this case the axial function takes values $\alpha(pq)$ along these directed line segments ${\dsp (1-t)f(p)+tf(q)}$.  Although most of the 1-skeleta shown in this paper are embedded, a general 1-skeleton need not admit an embedding at all; see Fig.~\ref{fig:FultonsEG} or Fig.~\ref{fig:NoPolarization}.

\subsubsection{Example:  1-skeleta from fans}
We assume a basic familiarity with fans here.  The reader is referred to \cite{F} or \cite{Z} for the necessary prerequisites.
Fix a complete simplicial fan $\Sigma\subset\left(\R^d\right)^*$, and denote by $\Sigma_k$ the set of $k$-cones of $\Sigma$.  Let $V_\Sigma\coloneqq\Sigma_d$ denote the set of $d$-cones, and let $E_\Sigma$ denote the set of ``oriented'' $(d-1)$-cones regarded as ordered intersections of $d$-cones, i.e. ${\dsp \tau=\sigma_1\cap\sigma_2}$ and ${\dsp \bar{\tau}=\sigma_2\cap\sigma_1}$.  The pair $(V_\Sigma,E_\Sigma)$ defines a $d$-valent graph $\Gamma_\Sigma$.  For each oriented $(d-1)$-cone $\tau=\sigma_1\cap\sigma_2$ choose and fix a normal vector $\alpha_\tau\in\R^d$ pointing inside $\sigma_1$, i.e. $\langle x,\alpha_\tau\rangle \ > \ 0$ for all $x\in\operatorname{int}(\sigma_1)$.  We can even choose our normal vectors so that $\alpha_\tau=-\alpha_{\bar{\tau}}$.  The claim is that the function 
$$\xymap{E_\Sigma\ar[r]^-{\alpha} & \R^d\\
\tau\ar@{|->}[r] & \alpha_\tau\\}$$   
is a $d$-independent axial function on $\Gamma_\Sigma$.  Clearly $\alpha$ is $d$-independent since the normal vectors of a $d$-cone are a basis for $\R^d$.  To see that A3 holds fix an oriented $(d-1)$-cone $\tau=\sigma_1\cap\sigma_2$ and any other ``oriented edge'' issuing from $\sigma_1$, say $\tau'=\sigma_1\cap\sigma_2'$.  Let $\rho'$ and $\rho''$ be rays (i.e. one-dimensional cones) such that $\tau+\rho^{(i)}=\sigma_i$ for $i=1,2$.  Then we have $\tau'=\tau\cap\tau'+\rho'$, and there is a unique $(d-1)$-cone defined by $\tau''=\tau\cap\tau'+\rho''$.  Define the connection map $\theta_\tau(\tau')\coloneqq\tau''$.  We observe that there are positive constants $\lambda_\tau(\tau')$ such that 
\begin{equation}
\label{eq:alphafan}
\alpha_{\tau'}-\lambda_\tau(\tau')\alpha_{\tau''}=c\cdot\alpha_\tau.
\end{equation}
Indeed the LHS of Eq. \eqref{eq:alphafan} necessarily vanishes on $\tau'\cap\tau''=\tau\cap\tau'$.  Let $\rho$ be the ray determined by $\tau=\tau\cap\tau'+\rho$, and choose any nonzero covector $\eta\in\rho$.  Then defining
\begin{equation}
\label{eq:compatibilityfan}
\lambda_\tau(\tau')\coloneqq\frac{\langle\eta,\alpha_{\tau'}\rangle}{\langle\eta,\alpha_{\tau''}\rangle}
\end{equation}
forces LHS of Eq. \eqref{eq:alphafan} to vanish on $\rho$ as well, hence on all of $\tau$.  Thus LHS of Eq. \eqref{eq:alphafan} must be a multiple of $\alpha_{\tau}$.  Thus the triple $(\Gamma_\Sigma,\alpha,\theta)$ defines a $d$-valent, $d$-independent 1-skeleton in $\R^d$.

A \emph{conewise-linear function} on $\Sigma$ is a continuous function $F\colon\left|\Sigma\right|\rightarrow\R$ whose restriction to every cone in $\Sigma$ is linear.  Write $F_{\sigma}$ for the linear function that $F$ restricts to on $\sigma$.  $F$ is called \emph{strictly convex} if for any two distinct cones $\sigma,\sigma'\in\Sigma$ and any $x\in\sigma$ we have $F_{\sigma'}(x)>F_{\sigma}(x)$.  A standard result states that a complete (simplicial) fan $\Sigma$ admitting a strictly convex conewise-linear function $F\colon\left|\Sigma\right|\rightarrow\R$ is the normal fan of a convex (simple) polytope $P$ given as the convex hull of the points $F_\sigma$.  For this reason, a complete fan that admits a strictly convex conewise-linear function is called \emph{polytopal}.  Note that for any oriented edge $\tau=\sigma_1\cap\sigma_2$ we have 
\begin{equation}
\label{eq:embeddingfan}
F_{\sigma_2}-F_{\sigma_1}=c_\tau\alpha_{\tau} 
\end{equation}
for some \emph{positive} scalar $c_\tau$.  Indeed since $F$ is continous, $F_{\sigma_2}$ and $F_{\sigma_1}$ must agree on $\tau$, hence the LHS of Eq. \eqref{eq:embeddingfan} is some multiple of $\alpha_{\tau}$.  By the strict convexity condition, for $x\in \sigma_1\setminus\tau$, we have $F_{\sigma_2}(x)-F_{\sigma_1}(x) \ > \ 0$ hence the LHS is necessarily a positive multiple of $\alpha_{\tau}$.  Thus a strictly convex conewise-linear function on $\Sigma$ defines an embedding for $(\Gamma_\Sigma,\alpha,\theta)$.  

\begin{remarks*}
\begin{enumerate}
\item The 1-skeleton corresponding to a complete simplicial polytopal fan is the 1-skeleton of a simple polytope in the usual sense.
\item There are examples of complete simplicial non-polytopal fans.  Fulton \cite[pg. 71]{F} gives a $3$-dimensional example associated to a smooth complete non-projective toric variety; Fig.~\ref{fig:FultonsEG} shows the corresponding 1-skeleton.  
\end{enumerate}
\end{remarks*}
\begin{figure}
  \includegraphics{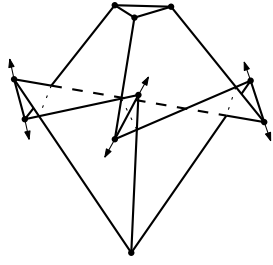}
\caption{No Embedding/Non-Polytopal}
\label{fig:FultonsEG}       
\end{figure}

\subsection{Subskeleta and Holonomy}
Fix a $d$-valent 1-skeleton $(\Gamma,\alpha,\theta)\subset\R^n$. 

Let $p$ and $q$ be vertices of $\Gamma$.  A \emph{path} from $p$ to $q$ is a sequence of vertices beginning with $p$ and ending with $q$ such that any two consecutive vertices in the path form an edge of $\Gamma$; a path from $p$ to $q$ is denoted by
$$\gamma\colon p\shortrightarrow\cdots\shortrightarrow q.$$  
A \emph{loop} is a path that begins and ends at the same vertex.

A \emph{subgraph} of $\Gamma$ is a regular graph $\Gamma_0=(V_0,E_0)$ where $V_0\subset V_\Gamma$ and $E_0\subset E_\Gamma$.  If the connection on $\Gamma$ restricts to $\Gamma_0$ in the sense that ${\dsp \theta_e\left(E_0^{i(e)}\right)\subseteq\left(E_0^{t(e)}\right)}$ for every edge $e\in E_0$ we say the subgraph is \emph{totally geodesic}.  The restriction of the connection $\theta$ and axial function $\alpha$ to the totally geodesic subgraph $\Gamma_0$ defines a 1-skeleton triple $(\Gamma_0,\alpha_0,\theta_0)$ called a \emph{subskeleton}.

The 1-skeleton $(\Gamma,\alpha,\theta)\subset\R^n$ always admits a certain class of subskeleta called \emph{$k$-slices}.  For each $k$-dimensional subspace $H\subset\R^n$ define $\Gamma_H$ to be any connected component of the graph $G=(V_G,E_G)$ determined by the condition that the edges' directions lie in $H$, i.e. ${\dsp E_G\coloneqq\left\{e\in E_\Gamma\left|\right.\alpha(e)\in H\right\}}$ and ${\dsp V_G\coloneqq\left\{p\in V_\Gamma\left|\right. p=i(e), \ e\in E_G\right\}.}$  The claim is that the subgraph $\Gamma_H$ has constant valency and is totally geodesic.  Indeed fix $e\coloneqq pq$ an oriented edge of $\Gamma_H$, and let $E_H^p$ denote the oriented edges at $p$ lying in $\Gamma_H$.  Let $e'\in E_H^p$ be any oriented edge at $p$ different from $e$.  Then $\alpha(\theta_e(e'))$ must lie in the $2$-plane generated by $\alpha(e)$ and $\alpha(e')$ by condition A3 in Definition \ref{def:axial}.  Since $\alpha(e)$ and $\alpha(e')$ both lie in the subspace $H$, $\alpha(\theta_e(e'))$ must also lie in $H$, hence $\theta_e(e')\in E_H^q$.  This shows that $\theta_e\left(E_H^p\right)\subseteq E_H^q$, and a similar argument shows that $\theta_{\bar{e}}\left(E_H^q\right)\subseteq E_H^p$.  Since these maps are inverses of one another we see that $\left|E_H^p\right|=\left|E_H^q\right|$, and, since $\Gamma_H$ is connected, it must therefore have constant valency.  

The $k$-faces of a simple polytope are $k$-valent $k$-slices of its corresponding 1-skeleton.  Note that a $k$-slice of a 1-skeleton need not be $k$-valent in general.

Fix a subskeleton $(\Gamma_0,\alpha_0,\theta_0)\subseteq(\Gamma,\alpha,\theta)$.  An oriented edge not in $E_0$ but whose initial vertex lies in $V_0$ is \emph{normal} to $\Gamma_0$.  The set of oriented normal edges to $\Gamma_0$ will be denoted by $N_0$.  For each vertex $p\in V_0$ the oriented edge set $E^p$ is partitioned into two pieces:  the oriented edges at $p$ in $\Gamma_0$, denoted $E^p_0$, and those at $p$ normal to $\Gamma_0$, denoted $N^p_0$.  Moreover, since $\Gamma_0$ is totally geodesic the connection also splits:  ${\dsp \theta=\theta_0\sqcup\theta^\perp.}$  The collection of maps ${\dsp \theta^\perp\coloneqq\left\{\theta^\perp_{pq}\colon N^p_0\rightarrow N_0^q\left|\right. pq\in E_0\right\}}$ define the \emph{normal connection} for the subskeleton.

Given any vertices $p,q\in V_0$ and a path joining them in $\Gamma_0$, $\gamma\colon p=p_0\sra\cdots\sra p_j=q$ define the \emph{path-connection map} for $\gamma$ to be the composition
$$K_\gamma\coloneqq\theta_{p_{j-1}p_j}\circ\ldots\circ\theta_{p_0p_1} \colon E_0^{p_0}\rightarrow E_0^{p_j}.$$
Define the \emph{normal path-connection map} for $\gamma$ to be the corresponding compositon of normal connection maps:
$$K_\gamma^\perp\coloneqq\theta_{p_{j-1}p_j}^\perp\circ\ldots\circ\theta_{p_0p_1}^\perp \colon N_0^{p_0}\rightarrow N_0^{p_j}.$$
For $\gamma$ a loop, the (resp. normal) path-connection map $K_\gamma$ gives a permutation of the set (resp. $N^p_0$) $E_0^p$, called a (resp. \emph{normal}) \emph{holonomy map}.

\begin{definition}
\label{def:trivialnormalholonomy}
A subskeleton $(\Gamma_0,\alpha_0,\theta_0)\subseteq(\Gamma,\alpha,\theta)$ has \emph{trivial normal holonomy} if the holonomy map $K_\gamma^\perp$ is trivial for all loops $\gamma\subset\Gamma_0$. 
\end{definition} 

The regular octahedron shown in Fig.~\ref{fig:NoLift} is a $4$-valent $3$-independent 1-skeleton whose triangular faces are $2$-valent subskeleta with non-trivial normal holonomy.  Indeed the normal holonomy map for the loop $\gamma\colon u\sra v\sra w\sra u$ applied to oriented edge $ux$ gives 
$$K^\perp_\gamma\colon ux\mapsto vx\mapsto wy\mapsto uy.$$   
\begin{figure}
\includegraphics{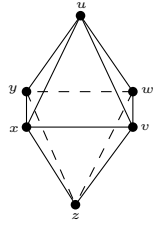}
\caption{Non-Trivial Normal Holonomy}
\label{fig:NoLift}
\end{figure}

On the other hand, the $k$-faces of a simple polytope always have trivial normal holonomy.  Indeed a $k$-face $\Gamma_0$ containing vertex $p$ is a $k$-slice for the $k$-dimensional subspace spanned by the set $\alpha(E^p_0)$.  Then for any loop $\gamma\colon p\sra p_1\sra\cdots\sra p_N\sra p$ in $\Gamma_0$, and for any edge $e\in N_0^p$, condition A3 implies that $\alpha(e)$ and $\alpha(K^\perp_\gamma(e))$ together with the spanning set $\alpha(E^p_0)$ span a $(k+1)$-dimensional subspace.  Thus by $d$-independence of the set $\alpha(E^p)$ we must have $K_\gamma^\perp(e)=e$.

\subsection{Polarizations and Combinatorial Betti Numbers}
An \emph{orientation} of $\Gamma$ is a choice of one orientation for each edge in $\Gamma$; the chosen oriented edge is called a \emph{directed edge} for the orientation.  A path 
$$\gamma\colon p\shortrightarrow\cdots\shortrightarrow q$$ is said to be \emph{oriented} (with respect to a given orientation on $\Gamma$) if $p_ip_{i+1}$ is a directed edge for all $i$.  The orientation is called \emph{acyclic} if there are no oriented loops.

A covector $\xi\in(\R^n)^*$ is \emph{generic} for the 1-skeleton $(\Gamma,\alpha,\theta)\subset\R^n$ if it pairs nonzero with each edge direction, i.e. $\langle\xi,\alpha(e)\rangle\neq 0$ for each $e\in E_\Gamma$.  A generic covector $\xi$ for $(\Gamma,\alpha,\theta)$ induces an orientation on $\Gamma$ by declaring the directed edges to be those oriented edges that pair positively with $\xi$.  See Fig.~\ref{fig:NotPointed2}.

\begin{definition}
\label{def:polarizingvector}
The generic covector $\xi$ is called \emph{polarizing} if it induces an acyclic orientation on $\Gamma$.  The 1-skeleton $(\Gamma,\alpha,\theta)$ \emph{admits a polarization} if it admits a generic polarizing covector $\xi$.
\end{definition}

\begin{figure}
	\includegraphics{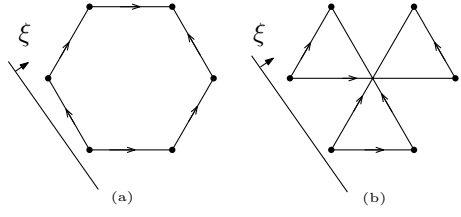}
\caption{Polarized 1-Skeleta}
\label{fig:NotPointed2}
\end{figure}

Note that a generic polarizing covector for $(\Gamma,\alpha,\theta)$ is also generic and polarizing for any subskeleton $(\Gamma_0,\alpha_0,\theta_0)$.

\begin{remark*}
In \cite{GZ1}, Guillemin and Zara use the term ``polarizing" to describe what we call ``generic" and what we call a ``polarizing covector" they call a ``polarizing covector satisfying the `no-cycle condition'".
\end{remark*}

A general 1-skeleton need not have a polarization at all, e.g. Fig.~\ref{fig:NoPolarization}.    
\begin{figure}
	\includegraphics{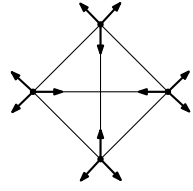}
\caption{No Polarization}
\label{fig:NoPolarization}
\end{figure}
On the other hand if a 1-skeleton $(\Gamma,\alpha,\theta)\subset\R^n$ admits an embedding, then every generic covector $\xi\in(\R^n)^*$ is polarizing.  Indeed an ``embedded'' oriented loop $\gamma\colon f(p_0)\sra f(p_1)\sra\cdots\sra f(p_N)\sra f(p_0)$ in $(\Gamma,\alpha,\theta)$ would give an inconsistent system of inequalities:
$$\langle \xi,f(p_0)\rangle \ > \ \langle\xi,f(p_1)\rangle \ > \ \cdots \ > \ \langle\xi, f(p_N)\rangle \ > \ \langle\xi,f(p_0)\rangle.$$
Guillemin and Zara also make this observation in \cite{GZ0}.

\begin{definition}
Given a polarizing covector $\xi\in(\R^n)^*$ for $(\Gamma,\alpha,\theta)$ we say an injective function $\phi\colon V_\Gamma\rightarrow\R$ is a \emph{Morse function} on $(\Gamma,\alpha,\theta)$ compatible with $\xi$ if for each oriented edge $pq\in E_\Gamma$ satisfying $\langle\xi,\alpha(pq)\rangle > \ 0$ we have $\phi(p)<\phi(q)$.
\end{definition}

As pointed out in \cite{GZ1}, the existence of a polarizing covector guarantees the existence of a compatible Morse function.  Indeed just define $\phi(p)$ to be the length of the longest oriented path in $\Gamma$ that starts at $p$.  This is well defined since there are no oriented loops.  We can then perturb $\phi$ a little to make it injective.

\begin{definition}
\label{def:indexbetti}
For $p\in V_\Gamma$ define the \emph{index} of $p$ (with respect to a generic covector $\xi$) to be the number of oriented edges ``flowing into" $p$; in other words
$$\operatorname{ind}_\xi(p)\coloneqq\# \{e\in E_p\mid \langle\xi,\alpha(e)\rangle<0\}.$$
Define the \emph{$i^{th}$ combinatorial Betti number} of $\Gamma$ by
$$b_i(\Gamma,\alpha)\coloneqq\#\{p\in V_\Gamma\mid \operatorname{ind}_\xi(p)=i\}.$$
\end{definition}

While the index of a vertex of $\Gamma$ clearly depends on the generic covector, $\xi$, an elegant argument due to Bolker shows that the combinatorial Betti numbers are actually independent of $\xi$.  See \cite[Theorem 1.3.1]{GZ1}. 

A vertex $p_0\in V_\Gamma$ with $\operatorname{ind}_\xi(p_0)=0$ is called a ($\xi$-) \emph{source} of $\Gamma$ for a generic covector $\xi$.  If $(\Gamma,\alpha,\theta)$ has a unique source for some (hence every) generic $\xi$, i.e. $b_0(\Gamma,\alpha)=1$, then we say that $(\Gamma,\alpha,\theta)$ is \emph{pointed}.  In Fig.~\ref{fig:NotPointed2}, the 1-skeleton labeled (a) is pointed, but the 1-skeleton labeled (b) is not.

Guillemin and Zara \cite{GZ1} introduced and studied a class of $3$-independent 1-skeleta called \emph{noncyclic}.  Here is their definition:
\begin{definition}
\label{def:noncyclic}
A $3$-independent 1-skeleton $(\Gamma,\alpha,\theta)\subset\R^n$ is called \emph{non-cyclic} if the following conditions hold:
\begin{enumerate}[{NC}1.]
\item $(\Gamma,\alpha,\theta)\subset\R^n$ admits a polarization   
\item Every $2$-slice is pointed.
\end{enumerate}
\end{definition}
The class of 1-skeleta coming from simple polytopes (or complete simplicial polytopal fans) is certainly noncyclic.  Indeed any generic covector is polarizing since the 1-skeleton has an embedding, and the $2$-slices are convex polygons, hence they are pointed.  Fig.~\ref{fig:PointedNotNoncyclic} shows a $3$-independent 1-skeleton that is not noncyclic.

\subsubsection{Example: Fans from 1-skeleta}
Let $\Sigma\subset\left(\R^d\right)^*$ be a complete simplicial $d$-dimensional fan with associated $d$-valent $d$-independent 1-skeleton $(\Gamma_\Sigma,\alpha,\theta)\subset\R^d$.  For any genric covector $\xi\in\left(\R^d\right)^*$ there must be a unique vertex $p_0\in V_\Gamma$ such that $\langle\xi,\alpha(e)\rangle \ > \ 0$ for all $e\in E^{p_0}$, corresponding to the unique $d$-cone of $\Sigma$ containing $\xi$.  Thus $(\Gamma_\Sigma,\alpha,\theta)$ is pointed.  It is straight forward to show that in fact every $k$-slice of $(\Gamma_\Sigma,\alpha,\theta)$ is pointed.  Indeed a $k$-slice $(\Gamma_H,\alpha_H,\theta_H)$ of $(\Gamma_\Sigma,\alpha,\theta)$ corresponds to the star of some $(d-k)$-cone $\tau$ of $\Sigma$, the boundary of which is identified with a complete simplicial $k$-dimensional fan via linear projection.  Then the projection of a generic covector along the linear span of $\tau$ is identified with the restriction of the generic covector to $H$, which must also be generic for $(\Gamma_H,\alpha_H,\theta_H)$.  A $d$-valent $d$-independent 1-skeleton will be  called \emph{toral} if every $k$-slice for $1\leq k\leq d$ is pointed.  
Hence there is a correspondence between complete simplicial fans and toral 1-skeleta
\begin{equation}
\label{eq:onetoone}
\xymapp{\left\{\text{Complete simplicial fans in $\left(\R^d\right)^*$}\right\}\ar[r] & \left\{\text{Toral 1-skeleta in $\R^d$}\right\}\\
\Sigma\ar@{|->}[r] & (\Gamma_\Sigma,\alpha,\theta).\\}
\end{equation}
The claim is that \eqref{eq:onetoone} is bijective.  

Let $(\Gamma,\alpha,\theta)$ denote a toral 1-skeleton in $\R^d$.  For each vertex $p$, define the simplicial polyhedral cone 
$$X_p\coloneqq\left\{u\in\left(\R^d\right)^*\left|\right. \langle u,\alpha(e)\rangle\geq 0, \ e\in E^p\right\}.$$
Let $\Sigma$ denote the set of cones $X_p$ and all their faces.  In order to establish that $\Sigma$ is a complete simplicial fan in $\left(\R^d\right)^*$ we need to show that
\begin{enumerate}
\item ${\dsp \bigcup_{p\in V_\Gamma}X_p=\left(\R^d\right)^*}$
\item ${\dsp X_p\cap X_q}$ is a face of both $X_p$ and $X_q$.
\end{enumerate}
To verify condition (i), note that the set of generic covectors for $(\Gamma,\alpha,\theta)$ is a dense open set $\mathcal{G}\subseteq\left(\R^d\right)^*$, and that each generic covector lies in some $X_p$.  Thus $\mathcal{G}\subseteq\bigcup_{p\in V_\Gamma}X_p\subseteq\left(\R^d\right)^*$, and taking closures yields (i).  To verify (ii), fix vertices $p$ and $q$, and let $\tau$ be the smallest face of $X_p$ containing $X_p\cap X_q$.  Define the $k$-dimensional subspace $H=\tau^\perp\coloneqq\left\{v\in\R^d\left|\right.\langle u,v\rangle=0 \ \forall \ u\in\tau\right\}$.  Let $\Gamma^p_H$ denote the graph of the $k$-slice containing $p$ and let $\Gamma^q_H$ denote the graph of the $k$-slice containing $q$.  The claim is that $\Gamma^p_H=\Gamma_H^q$.  Otherwise the two $k$-slices $(\Gamma^p_H,\alpha_H^p,\theta_H^p)$ and $(\Gamma^q_H,\alpha_H^q,\theta_H^q)$ are pointed by a common generic covector $\xi\in H^*$.  Let $p_0\in V_H$ and $q_0\in V'_H$ be the respective $\xi$-sources.  Note that there must exist a covector $\eta\in \tau$ such that $\langle \eta,\alpha(e)\rangle \ > \ 0$ for all $e\in N^{p_0}_H\cup N^{q_0}_H$.  Indeed if this were not the case, then there would exist at least one oriented edge $e\in N^{p_0}_H\cup N^{q_0}_H$ such that $\langle\eta,\alpha(e)\rangle =0$ for all $\eta\in \tau$, forcing $e$ to lie in $E_H\cap N_H$, which is impossible.  Extending $\xi\in H^*$ by zero to a covector in $\left(\R^d\right)^*$, and adding to it the covector $\eta$, we get a generic covector $\omega=\xi+\eta$ such that $\langle\omega,\alpha(e)\rangle \ > \ 0$ for all $e\in E^{p_0}\cup E^{q_0}$, which is impossible since $(\Gamma,\alpha,\theta)\subset\R^d$ is pointed.  Therefore $\Gamma_H^p=\Gamma_H^q=\Gamma_H$ is a single $k$-slice containing both $p$ and $q$.  

Now we want to show that ${\dsp \tau\subseteq X_p\cap X_q}$.  Fix $u\in \tau$.  Then $\langle u,\alpha(e)\rangle=0$ for all $e\in E^p_H$, and $\langle u,\alpha(e')\rangle \ > \ 0$ for all $e\in N_H^p$.  For $pp_1\in E_H$, and $e\in N^p_H$ we have by A3, $\alpha(e)-\lambda_{pp_1}(e)\alpha(\theta_{pp_1}(e))\in\R\cdot\alpha(pp_1)$.  Thus $\langle u,\alpha(e)\rangle=\lambda_{pp_1}(e)\langle u,\alpha(\theta_{pp_1}(e))\rangle \ > \ 0$ implies that $\langle u,\alpha(\theta_{pp_1}(e))\rangle \ > \ 0$.  Since $\Gamma_H$ is connected, we inductively conclude that $\langle u,\alpha(e'')\rangle \ > \ 0$ for all $e''\in N_H^q$.  Thus $u\in X_p\cap X_q$, hence $\tau=X_p\cap X_q$ is a face of $X_p$.  A similar argument shows that $\tau=X_p\cap X_q$ is a face of $X_q$.  This shows that $\Sigma$ is a complete simplicial $d$-dimensional fan.  

\begin{remarks*}
\begin{enumerate}
\item One can show that a noncyclic $3$-independent 1-skeleton is always pointed.  The converse, however, is not true in general.  Fig.~\ref{fig:PointedNotNoncyclic} shows a $3$-independent (embedded) pointed 1-skeleton that is not noncyclic.
\item It follows from (i) that a noncyclic $d$-valent, $d$-independent 1-skeleton is always toral.  In particular, the one-to-one correspondence in \eqref{eq:onetoone} restricts to a one-to-one correspondence between complete simplicial polytopal fans in $(\R^d)^*$ and $d$-valent $d$-independent noncyclic 1-skeleta admitting embeddings.
\item It is unclear whether every toral 1-skeleton is noncylic.  In other words, must every toral 1-skeleton admit a polarization?
\end{enumerate}
\end{remarks*}

\begin{figure}
	\includegraphics{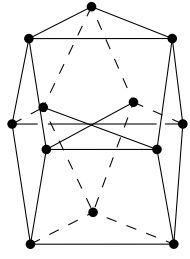}
\caption{Embedded, Pointed, Not Noncyclic}
\label{fig:PointedNotNoncyclic}
\end{figure}

\section{Generalizations}
\subsection{Generalized 1-Skeleta}
Fix a $d$-valent graph connection pair $(\Gamma,\theta)$.
\begin{definition}
\label{def:compatibilitysystem}
A \emph{compatibility system} for $(\Gamma,\theta)$ ${\dsp\lambda\coloneqq\{\lambda_e\}_{e\in E_\Gamma}}$ 
is a collection of maps ${\dsp\lambda_e\colon E_{i(e)}\rightarrow\R_+}$ indexed by the oriented edges of $\Gamma$ that satisfy the rule
$$\lambda_{\bar{e}}\circ\theta_e=\frac{1}{\lambda_e}$$
for every $e\in E_\Gamma$.
\end{definition}

The triple $(\Gamma,\theta,\lambda)$ defines a \emph{pre 1-skeleton}.

\begin{definition}
\label{def:generalaxial}
A \emph{generalized axial function} $\alpha$ for the pre-1-skeleton $(\Gamma,\theta,\lambda)$ is a map $\alpha\colon E_\Gamma\rightarrow\R^n$ that satisfies the following axioms:
\begin{enumerate}[{gA}1.]
\item For each $e\in E_\Gamma$ there is some $m_e>0$ such that $\alpha(e)=-m_e\alpha(\bar{e})$\\
\item For every $e\in E_\Gamma$ and each $e'\in E^{i(e)}\setminus\{e\}$ we have $$\alpha(e')-\lambda_e(e')\alpha(\theta_e(e'))\in\R\cdot\alpha(e).$$

\end{enumerate}
\end{definition}

The quadruple $(\Gamma,\alpha,\theta,\lambda)$ is a \emph{generalized 1-skeleton}.  Note that if the generalized axial function $\alpha$ is $2$-independent the compatibility system $\lambda$ is uniquely determined.  If, in addition, the constants $m_e$ are all equal to $1$ in condition gA1 then $\alpha$ is actually an axial function and the quadruple $(\Gamma,\alpha,\theta,\lambda)$ defines a 1-skeleton.  

It will be useful to have a notion of equivalence of generalized 1 skeleta.

\begin{definition}
\label{def:eqgenaxial}
Two generalized 1 skeleta $(\Gamma,\alpha,\theta,\lambda)\subset\R^n$ and $(\tilde{\Gamma},\tilde{\alpha},\tilde{\theta},\tilde{\lambda})\subset\R^n$ are equivalent if the graph connection pairs $(\Gamma,\theta)$ and $(\Gamma',\theta')$ are equal and there exists a function $\kappaup\colon E_\Gamma\rightarrow\R_+$ such that for every $e\in E_\Gamma$ and $e'\in E^{i(e)}$ we have
\begin{enumerate}
\item ${\dsp \alpha(e)=\kappaup(e)\cdot\alpha'(e),}$ and 
\item ${\dsp \lambda_e(e')=\frac{\kappaup(e')}{\kappaup(\theta_e(e'))}\tilde{\lambda}_{e}(e').}$
\end{enumerate}
We will denote equivalence of generalized 1-skeleta by
$$(\Gamma,\alpha,\theta,\lambda)\equiv(\tilde{\Gamma},\tilde{\alpha},\tilde{\theta},\tilde{\lambda})$$ 
\end{definition}

Note that a $2$-independent generalized 1-skeleton is equivalent to a 1-skeleton.  Indeed fix any orientation of $\Gamma$, and let $E^+_\Gamma$ denote the oriented edges oriented positively.  Then define ${\dsp\kappaup(e)=\begin{cases}
1 & \text{if $e\in E^+_\Gamma$}\\
\frac{1}{m_e} & \text{if $e\notin E^+_\Gamma$}\\
\end{cases}}.$
We will abuse the notation slightly and refer to a generalized 1-skeleton with a $2$-independent axial function as a 1-skeleton.

The notions of subskeleton, (resp. normal) path-connection map, and (resp. normal) holonomy have obvious generalized analogues. Additionally, the compatibility system of a generalized 1-skeleton allows one to assign positive scalars to (resp. normal) path connection maps.  Fix a $d$-valent generalized 1-skeleton $(\Gamma,\alpha,\theta,\lambda)$ and a $k$-valent subskeleton $(\Gamma_0,\alpha_0,\theta_0,\lambda_0)$.
\begin{definition}
\label{def:pathconnectionnumbers}
Let $\gamma\colon p=p_0\sra p_1\sra\cdots p_{j-1}\sra p_j=q$ be a path in $\Gamma_0$ joining vertices $p$ and $q$ in $V_0$.  The \emph{path-connection number} for $\gamma$ is the product
$$\left|K_{\gamma}\right|\coloneqq\left(\prod_{e\in E_0^{p_0}}\lambda_{p_0p_1}(e)\right)\cdots\left(\prod_{e\in E_0^{p_{j-1}}}\lambda_{p_{j-1}p_j}(e)\right).$$
The \emph{normal path-connection number} for $\gamma$ is the product
$$\left|K_{\gamma}^\perp\right|\coloneqq\left(\prod_{e\in N_0^{p_0}}\lambda_{p_0p_1}(e)\right)\cdots\left(\prod_{e\in N_0^{p_{j-1}}}\lambda_{p_{j-1}p_j}(e)\right).$$
For each $e\in E^{p}$ the \emph{local path-connection number} for $\gamma$ at $e$ is the product 
$$\left|K_\gamma(e)\right|\coloneqq\prod_{i=1}^j\lambda_{p_{i-1}p_i} (\theta_{p_{i-2}p_{i-1}}\circ\cdots\circ\theta_{p_0p_1}(e)).$$
If $\gamma$ is a loop, replace the term ``path-connection" with the term ``holonomy".
\end{definition}
Of particular interest will be those subskeleta whose local normal holonomy numbers are trivial.  More precisely,
\begin{definition}
\label{def:straightlevel}
A subskeleton $(\Gamma_0,\alpha_0,\theta_0,\lambda_0)$ of a generalized 1-skeleton $(\Gamma,\alpha,\theta,\lambda)$ is \emph{level} if for each $e\in N_0^p$ and every loop $\gamma$ in $\Gamma_0$ such that $K_\gamma(e)=e$, we have ${\dsp\left|K_\gamma(e)\right|=1}$.
\end{definition}

In Fig.~\ref{fig:2Bananas} the triangular subskeleton $x \sra y\sra z\sra x$ has trivial normal holonomy, but is not level.  
\begin{figure}
\includegraphics{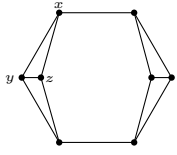}
\caption{Not Level}
\label{fig:2Bananas}
\end{figure}
On the other hand the $k$-faces of a simple polytope are always level.  Indeed for face $F$ and a fixed basepoint $p_0\in F$, let $e$ be any edge normal to $F$ at $p_0$.  Then there is a unique facet $H$ containing $F$ for which $e$ is normal, and $H$ corresponds to a $(d-1)$-slice $(\Gamma_H,\alpha_H,\theta_H)$.  Then for any choice of nonzero covector $\eta\in\left(\R^d\right)^*$ normal to the hyperplane $\spa(H)\subset\R^d$, the compatibility constants around the edges of $\Gamma_H$ are given by the formula ${\dsp \lambda_{pq}(e)=\frac{\langle\eta,\alpha(e)\rangle}{\langle\eta,\alpha(\theta_{pq}(e))\rangle}}$ as in Eq. \eqref{eq:compatibilityfan}.  Thus the local normal holonomy numbers associated to a loop $\gamma\colon p_0\sra p_1\sra\cdots\sra p_N\sra p_0$ are given by 
$$\left|K_\gamma^\perp(e)\right|=\frac{\langle\eta,\alpha(e)\rangle}{\langle\eta,\alpha(\theta_{p_0p_1}(e))\rangle}\cdots\frac{\langle\eta,\alpha(\theta_{p_{N-1}p_N}\circ\cdots\circ\theta_{p_0p_1}(e))\rangle}{\langle\eta,\alpha(K_\gamma^\perp(e))\rangle}=1.$$   

The notions of polarizations, compatible Morse functions, and combinatorial Betti numbers also have obvious generalized analogues.  We will see presently that the blow-up, reduction, and cutting constructions from \cite{GZ1} also have generalized analogues.  The exposition in \cite{GZ1} is closely followed; Subsections 3.2, 3.3.2, 3.3.3, and 3.4 are generalized analogues of, respectively, Subsections 2.2.1, 2.3.1, 2.3.2, and 2.4.1 in \cite{GZ1}.

\subsection{Blow-Up}
Fix a $d$-valent generalized 1-skeleton with connection $(\Gamma,\alpha,\theta,\lambda)\subset\R^n$ and let $(\Gamma_0,\alpha_0,\theta_0,\lambda_0)$ be a $k$-valent subskeleton.  We set up the following notation as a convention to be used throughout this subsection.
Let $p,q,r\in V_0$ denote arbitrary vertices with $pq,pr\in E_0$.  Let $e,e^\prime,e^{\prime\prime}\in N_0^p$ denote arbitrary edges at $p$ normal to $\Gamma_0$ and set $f,f^\prime,f^{\prime\prime}\in N_0^q$, and $g,g^\prime,g^{\prime\prime}\in N_0^r$ such that $\theta_{pq}(e^{(i)})=f^{(i)}$ and $\theta_{pr}(e^{(i)})=g^{(i)}$ for $0\leq i\leq 2$.
For a subset of oriented edges $E\subset E_\Gamma$ set $\overline{E}\coloneqq\left\{\bar{e}\left|\right.e\in E\right\}$.  

We will define a new graph $\Gamma^\sharp=(V^\sharp,E^\sharp)$ by cutting off $\Gamma_0$ in $\Gamma$ and replacing it with a new $(d-1)$-valent sub-graph.  Define the ``new vertex set'' ${\dsp V_0^\sharp\coloneqq\left\{z^p_e\left|\right. p\in V_0, \ e\in N_0^p\right\}.}$  Define the vertex set of $\Gamma^\sharp$ to be 
$$V^\sharp\coloneqq V_\Gamma\setminus V_0\sqcup V_0^\sharp.$$

Define the ``new edge sets'' ${\displaystyle N^\sharp_0\coloneqq\left\{z^p_e[t(e)]\left|\right.p\in V_0, \ e\in N_0^p\right\}}$ and ${\displaystyle E_0^\sharp\coloneqq\left\{z^p_ez^p_{e^\prime}\right\}\cup \left\{z^p_ez^q_f\right\}.}$  Define the edge set of $\Gamma^\sharp$ by 
$$E^\sharp=E_{\Gamma}\setminus\left(E_0\cup N_0\cup\overline{N}_0\right)\cup N^\sharp_0\cup \ol{N}^\sharp_0\cup E^\sharp_0.$$

Note that oriented edge sets $N_0^\sharp$ and $\ol{N}_0^\sharp$ are in one-to-one correspondence with $N_0$ and $\ol{N}_0$.  Thus it is clear that the vertices $x\in V_\Gamma\setminus V_0\subseteq V^\sharp$ (including vertices $x=[t(e)]$, $e\in N^0$) are incident to exactly $d$ oriented edges.  Also the vertices $z^p_e\in V_0^\sharp\subseteq V^\sharp$ are incident to the $d$ edges ${\dsp \left\{z^p_ez^p_{e^\prime}\left|\right. e^\prime\in N_0^p\setminus\left\{e\right\}\right\}\sqcup\left\{z^p_ez^q_f\left|\right. pq\in E_0\right\}\sqcup\left\{z^p_e[t(e)]\right\}.}$  Hence $\Gamma^\sharp$ is a $d$-valent graph called the \emph{blow-up graph} of $\Gamma$ along $\Gamma_0$. 

The natural map of sets ${\displaystyle \beta\colon V^\sharp\rightarrow V_\Gamma}$
\[\beta(x)=
\begin{cases}
q & \text{if $x=q\in V_\Gamma\setminus V_0$}\\
p & \text{if $x=z^p_e \ \ \text{for some $e\in N_p^0$}$}\\
\end{cases}
\]
extends to a map of graphs ${\dsp \beta\colon \Gamma^\sharp\rightarrow\Gamma}$ called the \emph{blow down map}.

The induced subgraph of $\Gamma^\sharp$ on the vertex set $\beta^{-1}(V_0)=V_0^\sharp$, denoted by $\Gamma_0^\sharp$, is a $(d-1)$-valent connected subgraph called the \emph{singular locus} of the blow-up.  

\begin{figure*}
  \includegraphics[width=0.75\textwidth]{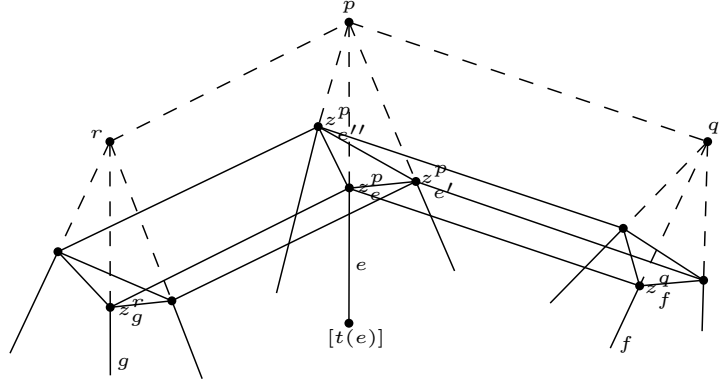}
\caption{Blow-Up Along a Subskeleton}
\label{fig:blowupill}       
\end{figure*}

For each vertex $x$ in the singular locus $V_0^\sharp$ define the \emph{horizontal edges at $x$} to be the edges in the singular locus at $x$ preserved by $\beta$, i.e. $(E_0^\sharp)_x^h\coloneqq\beta^{-1}(E^0_{\beta(x)})$.  Define the \emph{vertical edges at $x$} to be those edges destroyed by $\beta$, i.e. $(E^\sharp)_x^v\coloneqq \left(E_0^\sharp\right)_x\setminus(E_0^\sharp)_x^h$.  

We want to define a connection and compatibility system on $\Gamma^\sharp$.  To avoid confusion, we will use the letter $\epsilon$ to denote an edge in $\Gamma^\sharp$ and reserve the letter $e$ for an edge in $\Gamma$.  For oriented edges $\epsilon$ not issuing from the singular locus $\Gamma_0^\sharp$, define 
\begin{align}
\label{eq:thetanotlocus}
\theta^\sharp_\epsilon & \coloneqq \beta^{-1}\circ\theta_{\beta(\epsilon)}\circ\beta\\
\label{eq:lambdanotlocus}
\lambda^\sharp_\epsilon & \coloneqq\lambda_{\beta(\epsilon)}\circ\beta.
\end{align}

The values of $\theta^\sharp$ and $\lambda^\sharp$ on oriented edges in $N_0^\sharp$ and $E_0^\sharp$ are listed in Tables \ref{tab:buthetavalues} and \ref{tab:bulambdavalues}.  From Eqs. \eqref{eq:thetanotlocus}, \eqref{eq:lambdanotlocus}, and Tables \ref{tab:buthetavalues}, \ref{tab:bulambdavalues}, it is straight forward to check that the triple $(\Gamma^\sharp,\theta^\sharp,\lambda^\sharp)$ defines a pre 1-skeleton.

\begin{table}
\caption{$\theta^\sharp$ values}
\label{tab:buthetavalues}       
\begin{tabular}{llll}
\hline\noalign{\smallskip}
$\epsilon$ & $\theta_{\epsilon}^\sharp\left(z^p_e[t(e)]\right)$ & $\theta_{\epsilon}^\sharp\left(z^p_ez^r_{g}\right)$ & 
$\theta_{\epsilon}^\sharp\left(z^p_ez^p_{e^{\prime\prime}}\right)$\\
\noalign{\smallskip}\hline\noalign{\smallskip}
$z^p_e[t(e)]$ & --------- & $\beta^{-1}\circ\theta_e(pr)$ & $\beta^{-1}\circ\theta_e(e^{\prime\prime})$\\
\noalign{\smallskip}
$z^p_ez^q_{f}$ & $z^q_f[t(f)]$ & $z^q_fz^s_h$ & $z^q_fz^q_{f^{\prime\prime}}$\\ 
\noalign{\smallskip}
$z^p_ez^p_{e^\prime}$ & $z^p_{e^\prime}[t(e^\prime)]$ & $z^p_{e^\prime}z^r_{g^\prime}$ & $z^p_{e^{\prime}}z^p_{e^{\prime\prime}}$\\
\noalign{\smallskip}\hline
\end{tabular}
\end{table}

\begin{table}
\caption{$\lambda^\sharp$ values}
\label{tab:bulambdavalues}       
\begin{tabular}{llll}
\hline\noalign{\smallskip}
$\epsilon$ & $\lambda_{\epsilon}^\sharp\left(z^p_e[t(e)]\right)$ & $\lambda_{\epsilon}^\sharp\left(z^p_ez^r_{g}\right)$ & 
$\lambda_{\epsilon}^\sharp\left(z^p_ez^p_{e^{\prime\prime}}\right)$\\
\noalign{\smallskip}\hline\noalign{\smallskip}
$z^p_e[t(e)]$ & -- & $\lambda_e(pr)$ & $\lambda_e\left(e^{\prime\prime}\right)$\\
\noalign{\smallskip}
$z^p_ez^q_{f}$ & $1$ & $\lambda_{pq}(pr)$ & $\lambda_{pq}\left(e^{\prime\prime}\right)$\\ 
\noalign{\smallskip}
$z^p_ez^p_{e^\prime}$ & $1$ & $1$ & $1$\\
\noalign{\smallskip}\hline
\end{tabular}
\end{table}

On oriented edges not issuing from the singular locus of $\Gamma^\sharp$, define
\begin{equation}
\label{eq:alphanotlocus}
\alpha^\sharp\coloneqq\alpha\circ\beta
\end{equation}
In order to extend the function $\alpha^\sharp$ to the remaining oriented edges issuing from the singular locus, we must assume the existence of a system of positive scalar assignments to the normal edges of $\Gamma_0$, ${\dsp n\colon N^0\rightarrow\R_+}$, satisfying the following compatibility condition:
\begin{equation}
\label{eq:E.6}
\frac{n(e')}{n(\theta_e(e'))}=\lambda_e(e') \ \ \text{for all $e\in\Gamma_0$ and all $e'\in N_0^{i(e)}$};
\end{equation}
such a system is called a \emph{blow-up system} for the subskeleton $(\Gamma_0,\alpha_0,\theta_0,\lambda_0)\subseteq (\Gamma,\alpha,\theta,\lambda)$.  It is worth remarking that the levelness of $(\Gamma_0,\alpha_0,\theta_0,\lambda_0)$ guarantees the existence of a blow-up system.  Indeed fixing a basepoint $p_0\in V_0$ and setting $n(e)\equiv 1$ for all $e\in N_0^{p_0}$, one can use the normal path connection and the normal path connection numbers to extend $n$ to $N^x_0$ for any other $x\in V_0$:  Simply take any path $\gamma\colon p_0\sra\cdots\sra x$ joining $p_0$ to $x$ in $\Gamma_0$.  Then for any $e\in N_0^x$ there is $\tilde{e}\in N_0^{p_0}$ such that $e=K_\gamma^\perp(\tilde{e})$.  Thus the function $$n(e)\coloneqq\frac{1}{\left|K_\gamma^\perp(\tilde{e})\right|}$$ is independent of the path $\gamma$, and hence defines the required blow-up system.   

We use a fixed blow-up system, $n\colon N^0\rightarrow\R_+$, for $(\Gamma_0,\alpha_0,\theta_0,\lambda_0)$ to define $\alpha^\sharp$ on $E_0^\sharp\sqcup N_0^\sharp$ as shown in Table \ref{tab:bualphavalues}.

\begin{table}
\caption{$\alpha^\sharp$ values}
\label{tab:bualphavalues}       
\begin{tabular}{llll}
\hline\noalign{\smallskip}
$\epsilon$ & $\alpha^\sharp(\epsilon)$\\
\noalign{\smallskip}\hline\noalign{\smallskip}
$z^p_e[t(e)]$ & $\frac{1}{n(e)}\alpha(e)$\\
\noalign{\smallskip}
$z^p_ez^q_{f}$ & $\alpha(pq)$\\ 
\noalign{\smallskip}
$z^p_ez^p_{e^\prime}$ & $\alpha(e^\prime)-\frac{n(e^\prime)}{n(e)}\alpha(e)$\\
\noalign{\smallskip}\hline
\end{tabular}
\end{table}

Clearly $\alpha^\sharp$ satisfies gA1 in Definition \ref{def:generalaxial}.  It remains to show that $\alpha^\sharp$ satisfies gA2: 
\begin{equation}
\label{eq:blowupgA2}
\alpha^\sharp(\epsilon^\prime)-\lambda^\sharp_{\epsilon}(\epsilon^\prime)\alpha^\sharp(\theta_{\epsilon}^\sharp(\epsilon^\prime))\equiv 0 \ \ \text{mod $\alpha^\sharp(\epsilon)$.}
\end{equation}

It is straight forward to verify gA2 for oriented edges not issuing from the singular locus.  Indeed in this case Eq. \eqref{eq:blowupgA2} becomes
$$\alpha(\beta(\epsilon'))-\lambda_{\beta(\epsilon)}(\beta(\epsilon'))\alpha(\theta_{\beta(\epsilon)}(\beta(\epsilon')))=c\alpha(\beta(\epsilon)),$$
which holds since $\alpha$ is a generalized axial function for the pair $(\Gamma,\theta)$.  Thus it suffices to verify Eq. \eqref{eq:blowupgA2} for oriented edges issuing from the singular locus.

For $\epsilon= z^p_ez^q_{f}$ and $\epsilon^\prime= z^p_ez^p_{e^\prime}$ the LHS of Eq. \eqref{eq:blowupgA2} becomes
\begin{equation}
\label{eq:ga21}
\left(\alpha(e^\prime)-\frac{n(e^\prime)}{n(e)}\alpha(e)\right)-\lambda_{pq}(e^\prime)\cdot\left(\alpha(f^\prime)-\frac{n(f^\prime)}{n(f)}\alpha(f)\right).
\end{equation}
Regrouping terms and using the identity ${\displaystyle \lambda_{pq}(e^\prime)\frac{n(f^\prime)}{n(f)}=\frac{n(e^\prime)}{n(f)}}$ derived from Eq. \eqref{eq:E.6}, Eq. \eqref{eq:ga21} becomes
\begin{equation}
\label{eq:ga23}
\left[\alpha(e^\prime)-\lambda_{pq}(e^\prime)\alpha(f^\prime)\right]-\left[\frac{n(e^\prime)}{n(e)}\alpha(e)-\frac{n(e^\prime)}{n(f)}\alpha(f)\right].
\end{equation}
Clearly the first term in \eqref{eq:ga23} is a multiple of $\alpha^\sharp(z^p_ez^q_{f})=\alpha(pq)$.  The second term is also a multiple of $\alpha(pq)$ by virtue of the condition ${\dsp \frac{n(e)}{n(f)}=\lambda_{pq}(e)}$. 

The remaining cases are straight forward and their verification is left to the reader.
Thus $\alpha^\sharp$ is a generalized axial function for the pre 1-skeleton $(\Gamma^\sharp,\theta^\sharp,\lambda^\sharp)$.

\begin{definition}
\label{def:blowup1sk}
The generalized 1-skeleton $(\Gamma^\sharp,\alpha^\sharp,\theta^\sharp,\lambda^\sharp)$ is called the \emph{blow-up} of $(\Gamma,\alpha,\theta,\lambda)$ along the subskeleton $(\Gamma_0,\alpha_0,\theta_0,\lambda_0)$.  
\end{definition}

\begin{remark*}
The blow-up construction for 1-skeleta was introduced by Guillemin and Zara \cite{GZ1}.  However the assumptions on the 1-skeleta made in \cite{GZ1} are a bit more restrictive than those we make here.  In particular they assume that
\begin{enumerate}[1.]
\item $\alpha$ is $3$-independent and
\item the compatibility constants along the normal edges $N_0$ are all equal to 1; i.e. $\lambda_{e}(e')=1$ for $e\in E_0$ and $e'\in N_0^{i(e)}$.
\end{enumerate}    
Note that if condition 2 holds then $(\Gamma_0,\alpha_0,\theta_0,\lambda_0)$ is level and if condition 1 holds then the function $\alpha^\sharp$ is actually $2$-independent.  
\end{remark*}

\subsection{Reduction}
\subsubsection{$2$-faces and reducible 1-skeleta}
For $3$-independent non-cyclic 1-skeleta, every $2$-slice is a convex polygon and, in \cite{GZ1}, these polygons are used to define the edges of cross sections.  Unfortunately without the $3$-independence condition some of these polygons may not show up as $2$-slices.  Fortunately we can recover these ``hidden'' polygons using the connection and a polarization.

\begin{definition}
\label{def:faces}  
A \emph{$k$-face} of a (polarized) generalized 1-skeleton $(\Gamma,\alpha,\theta,\lambda)$ is a $k$-valent subskeleton $(\Gamma_0,\alpha_0,\theta_0,\lambda_0)$ with $b_0(\Gamma_0,\alpha_0)=1$.  
\end{definition}

For example the $2$-faces of a $3$-independent noncyclic 1-skeleton are exactly its $2$-valent $2$-slices. 

\begin{definition}
\label{def:enoughfaces}
We say that a (polarized) generalized 1-skeleton $(\Gamma,\alpha,\theta,\lambda)$ has \emph{enough $k$-faces} if for each vertex $p\in V_\Gamma$ and any subset of $k$ edges $\{e_1,\ldots,e_k\}\in E^p$, there is a unique $k$-face containing $\{e_1,\ldots, e_k\}$.
\end{definition}

The 1-skeleton shown in Fig.~\ref{fig:NotReducible} equipped with the connection making the outer hexagon a $2$-face, does not have enough $2$-faces.  On the other hand the 1-skeleton of a simple polytope always has enough $k$-faces.  
\begin{figure}
\includegraphics{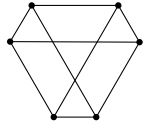}
\caption{Not Enough $2$-Faces}
\label{fig:NotReducible}
\end{figure}

\begin{definition}
\label{def:reducible}
A 1-skeleton $(\Gamma,\alpha,\theta)\subset\R^n$ is called \emph{reducible} if  
\begin{enumerate}[{R}1.]
\item it admits a polarization and  
\item it has enough $2$-faces.
\end{enumerate}
\end{definition} 

In the $3$-independent case the notion of noncyclic and reducible coincide.  We will see presently that the reduction construction from \cite{GZ1} applies to reducible 1-skeleta, and yields cross sections that are generalized 1-skeleta.    

\subsubsection{Cross sections}
For what follows it will be useful to keep track of the compatibility system of a 1-skeleton, so we shall include this in the notation.
Fix $(\Gamma,\alpha,\theta,\lambda)$ a $d$-valent reducible 1-skeleton in $\R^n$ with generic polarizing covector $\xi\in\left(\R^n\right)^*$.  Fix a $\xi$-compatible Morse function $\phi\colon V_\Gamma\rightarrow\R$ and a $\phi$-regular value $c\in\R$.  

A $2$-face $Q$ is a loop, which comes with two distinct orientations:\\
For a fixed basepoint $p_0\in Q$ we have ${\dsp \begin{cases}
Q\coloneqq\{p_0\sra p_1\sra\cdots\sra p_N\sra p_0\}\\
\bar{Q}\coloneqq\{p_0\sra p_N\sra\cdots\sra p_1\sra p_0\}.\\
\end{cases}}$  
An \emph{oriented} $2$-face is a $2$-face with a fixed orientation.  Let $\mathcal{F}_2$ denote the set of oriented $2$-faces of $(\Gamma,\alpha,\theta,\lambda)$.  Each oriented $2$-face $Q$ comes with a $\xi$-maximum vertex and a $\xi$-minimum vertex: 
$$M_\xi(Q)=\max^{}_{v\in Q}(\phi(v)) \ \ \ \ \text{and} \ \ \ \ \ m_\xi(Q)=\min^{}_{v\in Q}(\phi(v)).$$

Define the $c$-vertex set $V_c$ to be the oriented edges of $\Gamma$ at \emph{$c$-level}: $V_c=\left\{pq\in E_\Gamma\left|\right.\phi(p)<c<\phi(q)\right\}.$
Define the $c$-oriented edge set $E_c$ to be the oriented $2$-faces of $\Gamma$ at $c$-level:  $E_c=\left\{Q\in\mathcal{F}_2\left|\right. m_\xi(Q) \ < \ c \ < \ M_\xi(Q)\right\}.$
The condition $b_0(Q)=1$ implies that exactly two oriented edges in $Q$ lie at $c$-level.  Moreover exactly one of these oriented edges is a directed edge (via $\xi$) compatible  with the orientation on $Q$; this directed edge is the initial $c$-vertex of the oriented $c$-edge $Q$.  Hence the set $E_c$ consists of ordered pairs of $c$-vertices, and the pair $(V_c,E_c)$ defines a $(d-1)$-valent graph denoted by $\Gamma_c$.

The normal path-connection maps are used to define a connection on $\Gamma_c$.  In fact there are two natural connections on $\Gamma_c$ corresponding to the two paths around $Q$ joining its ``$c$-vertices''.

Fix $Q\in E_c$ and suppose $i(Q)=pq$ and $t(Q)=vw$ as in Fig.~\ref{fig:crossill}.  Define the \emph{upper path} from $q$ to $w$ 
\begin{equation}
\label{eq:uppath}
\gamma_{Q}^u\colon q=r_1\shortrightarrow r_2\shortrightarrow\ldots\shortrightarrow r_{k-1}\shortrightarrow r_k=w,
\end{equation}
meaning that $\phi(r_i)>c$ for $1\leq i\leq k$.  Analogously define the \emph{lower path} from $p$ to $v$ 
\begin{equation}
\label{eq:downpath}
\gamma_{Q}^d\colon p=t_1\shortrightarrow t_2\shortrightarrow\ldots\shortrightarrow t_{m-1}\shortrightarrow t_m=v,
\end{equation} 
meaning that $\phi(t_j)<c$ for $1\leq j\leq m$.

Since $(\Gamma,\alpha,\theta,\lambda)$ has enough $2$-faces, the oriented edges normal to $Q$, $e\in N^q_Q$, are in one-to-one correspondence with oriented $c$-edges distinct from $Q$, $R_e\in\left(E_c\right)_{pq}\setminus\left\{Q\right\}$ .  Thus the normal path-connection maps for the upper path (resp. lower path) define the \emph{up connection} (resp. \emph{down connection}) maps for $\Gamma_c$:
\begin{equation}
\label{eq:upconnection}
\xymap{(E_c)_{pq}\setminus\{Q\}\ar[r]^-{(\theta_c^u)_Q}\ar[d]_-\cong & (E_c)_{vw}\setminus\{\bar{Q}\}\ar[d]^-\cong\\
N_0^q\ar[r]_-{K^\perp_{\gamma_{Q}^u}} & N_0^w.\\}
\end{equation}
The down connection maps are defined analogously, replacing $K^\perp_{\gamma_Q^u}$ with ${\dsp K^\perp_{\gamma_Q^d}\colon N^p_0\rightarrow N^v_0.}$  Clearly the maps $\left(\theta_c^u\right)_Q$ (resp. $\left(\theta_c^d\right)_Q$) define a connection $\theta_c^u$ (resp. $\theta_c^d$) on $\Gamma_c$ called the \emph{up connection} (resp. \emph{down connection}).  See Fig.~\ref{fig:crossill}.

\begin{figure*}
  \includegraphics[width=0.75\textwidth]{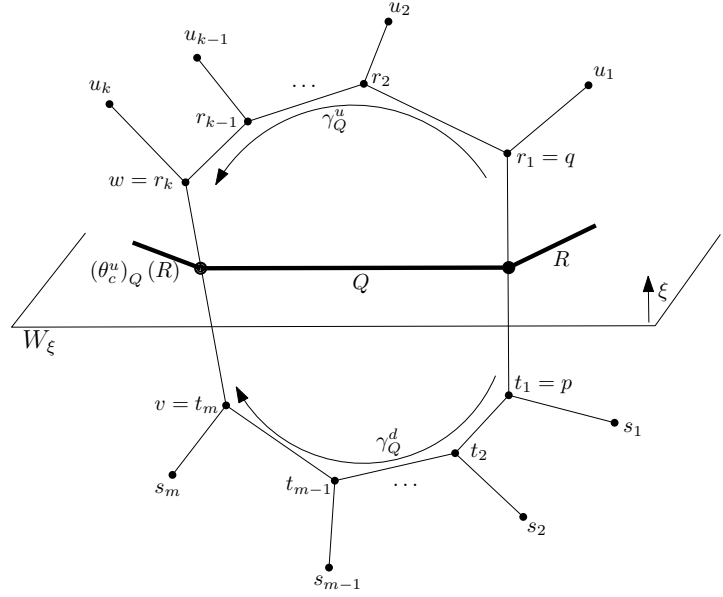}
\caption{$c$-Cross Section}
\label{fig:crossill}       
\end{figure*}

Just as we used path-connection maps to define the connection on $\Gamma_c$, we use the path-connection numbers to define a compatibility system on $\Gamma_c$.
Define the up compatibility maps by:
\begin{equation}
\label{eq:upcompatibility}
\xymap{(E_c)_{pq}\setminus\{Q\}\ar[dd]_-\cong\ar[dr]^-{(\lambda_c^u)_Q} & \\
& \R_+\\
N_0^q\ar[ur]_-{\left|K_{\gamma_{Q}^u}(-)\right|} & \\}
\end{equation}
where the lower map is defined by $e\mapsto |K_{\gamma_{Q}^u}(e)|$ as in Definition \ref{def:pathconnectionnumbers}.
Define the down compatibility maps analogously, replacing $\left|K_{\gamma_Q^u}(-)\right|$ with $\left|K_{\gamma_Q^d}(-)\right|\colon N^p_0\rightarrow \R^+.$ 

Observe that for every oriented $c$-edge $Q\in E_c$ and every $R_e\in\left(E_c\right)_{pq}$ we have  
$$(\lambda_c^u)_{\bar{Q}}\circ(\theta_c^u)_Q\left(R_e\right)=
\left|K_{\bar{\gamma}_{Q}^u}(K_{\gamma_{Q}^u}(e))\right|=\frac{1}{\left|K_{\gamma_{Q}^u}(e)\right|}=
\frac{1}{(\lambda_c^u)_Q\left(R_e\right)},$$
hence $\lambda_c^u$ defines a compatiblity system for the pair $(\Gamma_c,\theta_c)$, and similarly for $\lambda_c^d$.  

Therefore we have two (possibly distinct) pre 1-skeleta with the same underlying graph $\Gamma_c$, namely
${\dsp(\Gamma_c,\theta_c^u,\lambda_c^u)}$ 
and
${\dsp(\Gamma_c,\theta_c^d,\lambda_c^d).}$

For each pre-1-skeleton defined above, we can define a compatible, generalized axial function on $\Gamma_c$ as follows.  Let $W_\xi\subset\R^n$ denote the subspace annihilated by $\xi$.  Denote by $\textstyle{\bigwedge}^2\R^n$ the vector space of alternating two tensors generated by elements of the form $x\wedge y(=-y\wedge x)$ for $x,y\in\R^n$.  Let $\iota\colon\textstyle{\bigwedge}^2\R^n\rightarrow\R^n$ denote the $\xi$-interior product map defined by $\iota(x\wedge y)=\langle\xi,x\rangle y-\langle\xi,y\rangle x$.  As above let $Q\in E_c$ be an oriented $c$-edge with $i(Q)=pq$ and $t(Q)=vw$.  Let $$\gamma_j^u\colon q=r_1\shortrightarrow r_2\shortrightarrow\ldots\shortrightarrow r_j$$
be the partial upper path in $Q$ from $q$ to $r_j$ and
$$\gamma_j^d\colon p=t_1\shortrightarrow t_2\shortrightarrow\ldots\shortrightarrow t_j$$ 
the partial lower path in $Q$ from $p$ to $t_j$.  Then $\gamma^u_k=\gamma_{Q}^u\colon q\sra\cdots\sra r_k=w$ and $\gamma_m^d=\gamma_{Q}^d\colon p\sra\cdots\sra t_m=v$ 
as above.  Our convention will be to let $p=r_0$ and $v=r_{k+1}$ and to let $q=t_0$ and $w=t_{m+1}$.

Define the function ${\dsp\alpha_c^u\colon E_c\rightarrow W_\xi}$ by 
\begin{equation}
\label{eq:upaxial}
\alpha_c^u(Q)=\frac{\iota(\alpha(r_1r_0)\wedge\alpha(r_1r_2))}{\langle\xi,\alpha(r_1r_0)\rangle}.
\end{equation}
Similarly define the function ${\dsp\alpha_c^d\colon E_c\rightarrow W_\xi}$ by
\begin{equation}
\label{eq:downaxial}
\alpha_c^d(Q)=\frac{\iota(\alpha(t_1t_0)\wedge\alpha(t_1t_2))}{\langle\xi,\alpha(t_1t_0)\rangle}.
\end{equation}
An elegant argument due to Guillemin and Zara \cite[Theorem 2.3.1]{GZ1} applies verbatim to show that the function $\alpha_c^u$ (resp.~$\alpha_c^d$) defines a generalized axial function for the pre 1-skeleton $(\Gamma_c,\theta_c^u,\lambda_c^u)$ (resp.~$(\Gamma_c,\theta_c^d,\lambda_c^d)$).  We refer the reader to \cite{GZ1} for the details.

Thus we get two (possibly distinct) generalized 1-skeleta structures on the ($d-1$)-valent graph $\Gamma_c$; the \emph{up $c$-cross-section} of $\Gamma$, $(\Gamma_c,\alpha_c^u,\theta_c^u,\lambda_c^u)$, and the \emph{down $c$-cross-section} of $\Gamma$, $(\Gamma_c,\alpha_c^d,\theta_c^d,\lambda_c^d)$.  

An important component to the reduction technique is understanding what happens to cross sections as they pass over critical values.  A beautiful description in the noncyclic case involving a blow-up construction was given in \cite{GZ1}; it turns out such a description is also valid in the reducible case.

\subsubsection{Passage over a critical value}
Let $(\Gamma,\alpha,\theta,\lambda)$ be a $d$-valent reducible 1-skeleton in $\R^n$.  Fix a polarizing covector $\xi\in(\R^n)^*$ and a  $\xi$-compatible Morse function $\phi$.  Fix two $\phi$-regular values $c \ < \ c^\prime$ such that there is a unique vertex $p\in V_\Gamma$ satisfying $c<\phi(p)<c'$.  

Suppose that $\operatorname{ind}_\xi(p)=r$ and let 
$$V_{c,0}\coloneqq\left\{p_ip\left|\right. 1\leq i\leq r\right\}$$ 
denote those edges flowing into $p$, i.e. $\langle\xi,\alpha(p_ip)\rangle>0$ for $1\leq i\leq r$.  Let 
$$V_{c',0}\coloneqq\left\{pq_a\left|\right. 1\leq a\leq d-r\right\}$$
denote the oriented edges flowing out of $p$, i.e. $\langle\xi,\alpha(pq_a)\rangle \ > \ 0$ for $1\leq a\leq d-r$.  

Consider the up $c$-cross section $(\Gamma_c,\alpha_c^u,\theta_c^u,\lambda_c^u)$.  The set of oriented edges 
$V_{c,0}\subset V_c$ 
is the $c$-vertex set of a totally geodesic complete subgraph $\Gamma_{c,0}\subset\Gamma_c$.  Denote by $Q_{ia}$ the oriented $2$-face spanned by oriented edges $pp_i$ and $pq_a$ with initial $c$-vertex $pp_i$.  The set of oriented $c$-edges normal to $\Gamma_{c,0}$ is denoted by $N_{c,0}\coloneqq\{Q_{ia}\mid 1\leq a\leq (d-r) \ \ 1\leq i\leq r\}.$
The function
\begin{equation}
\label{eq:upblnumb}
\xymap{N_{c,0}\ar[r]^-{n^u} & \R_+\\
Q_{ia}\ar@{|->}[r]& \langle\xi,\alpha(pq_a)\rangle\\}
\end{equation}
defines a blow-up system for $(\Gamma_c^u,\alpha_c^u,\theta_c^u,\lambda_c^u)$ along $(\Gamma_{c,0}^u,\alpha_{c,0}^u,\theta_{c,0}^u,\lambda_{c,0}^u).$  Indeed for all $i,j, \ \text{and} \ a$ we have
\begin{equation}
\label{eq:blupcon}
\frac{n^u\left(Q_{ia}\right)}{n^u\left(\left(\theta^u_c\right)_{Q_{ij}}\left(Q_{ia}\right)\right)}= \frac{n^u\left(Q_{ia}\right)}{n^u\left(Q_{ja}\right)}= 1=\left(\lambda_c^u\right)_{Q_{ij}}\left(Q_{ia}\right).
\end{equation}
We can therefore blow-up along the subskeleton to get the generalized 1-skeleton 
${\dsp(\Gamma_c^\sharp,(\alpha_c^u)^\sharp,(\theta_c^u)^\sharp,(\lambda_c^u)^\sharp).}$

Similarly we can blow-up the down $c^\prime$-cross section $(\Gamma_{c'},\alpha_{c'}^d,\theta_{c'}^d,\lambda_{c'}^d)$ along the $(d-r)$-valent totally geodesic subskeleton $(\Gamma_{c',0},\alpha_{c',0}^d,\theta_{c',0}^d,\lambda_{c',0}^d)$.  Here the $c'$-vertex set of $\Gamma_{c',0}$ is the oriented edge set $V_{c',0}$ and the oriented $c'$-edges normal to $\Gamma_{c',0}$ are the oriented $2$-faces spanned by oriented edges $pp_i$ and $pq_a$, denoted by $Q_{ai}$ with $i(Q_{ai})=pq_a$.  In this case the function
\begin{equation}
\label{eq:downblnumb}
\xymap{N_{c',0}\ar[r]^-{n^d} & \R_+\\
Q_{ai}\ar@{|->}[r] & \langle\xi,\alpha(p_ip)\rangle\\}
\end{equation}
forms a blow-up system for $(\Gamma_{c'},\alpha_{c'}^d,\theta_{c'}^d,\lambda_{c'}^d)$ along the subskeleton $(\Gamma_{c',0},\alpha_{c',0}^d,\theta_{c',0}^d,\lambda_{c',0}^d)$ for similar reasons as above:  for all $a, \ b, \ \text{and} \ i$ we have
\begin{equation}
\label{eq:blupcon}
\frac{n^d\left(Q_{ai}\right)}{n^u\left(\left(\theta^d_{c'}\right)_{Q_{ab}}\left(Q_{ai}\right)\right)}= \frac{n^d\left(Q_{ai}\right)}{n^d\left(Q_{bi}\right)}= 1=\left(\lambda_c^d\right)_{Q_{ab}}\left(Q_{ai}\right).
\end{equation}
We can therefore blow-up along the subskeleton to get the generalized 1-skeleton 
${\dsp(\Gamma_{c'}^\sharp,(\alpha_{c'}^d)^\sharp,(\theta_{c'}^d)^\sharp,(\lambda_{c'}^d)^\sharp).}$
See Figure~\ref{fig:critpt2}.

\begin{figure*}
  \includegraphics[width=0.75\textwidth]{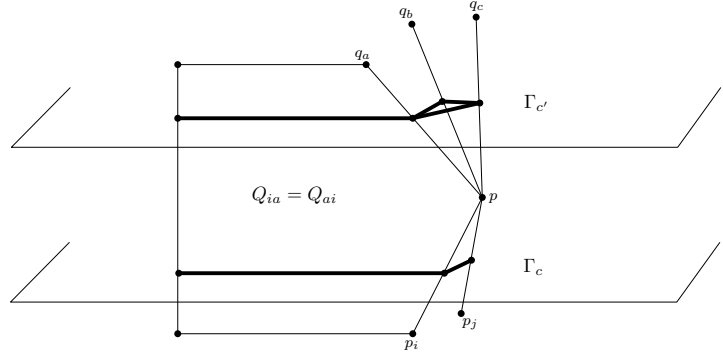}
\caption{Passage Over a Critical Point}
\label{fig:critpt2}       
\end{figure*}

In \cite{GZ1} Guillemin and Zara prove that for $3$-independent $\alpha$, the pairs $(\Gamma_c^\sharp,\left(\alpha_c^u\right)^\sharp)$, $(\Gamma_{c'}^\sharp,\left(\alpha_{c'}^d\right)^\sharp)$ are  ``equivalent'' in the sense that the graphs are isomorphic and the axial functions are positive multiples of one another.  Unfortunately this statement fails to hold without further assumptions when one considers the whole generalized 1-skeleta quadruple including connections and compatibility systems.  The following lemma is an analogue of Theorem 2.3.2 in \cite{GZ1}.

\begin{lemma}
\label{lem:equivblowup}
Assume that all $2$-faces in $E_c\cup E_{c'}$ are level and have trivial normal holonomy.  Then we have the following equivalences of 1-skeleta.
\begin{enumerate}[(i)]
\item ${\dsp (\Gamma_c,\alpha_c^u,\theta_c^u,\lambda_c^u)\equiv(\Gamma_c,\alpha_c^d,\theta_c^d,\lambda^d_c)}$.   
\item $(\Gamma_c^\sharp,(\alpha_c^u)^\sharp,(\theta_c^u)^\sharp,(\lambda_c^u)^\sharp) \equiv (\Gamma_{c'}^\sharp,(\alpha_{c'}^d)^\sharp,(\theta_{c'}^d)^\sharp,(\lambda_{c'}^d)^\sharp)$.
\end{enumerate}
\end{lemma}

\begin{proof}
First we establish (i).
Fix oriented $2$-faces $Q$ and $R$ as shown in Fig.~\ref{fig:crossill}, let $\gamma_Q^u$ (resp. $\gamma_Q^d$) be as in Eq. \eqref{eq:uppath} (resp. \eqref{eq:downpath}).  Identify the oriented $c$-edges at $i(Q)=pq$ with the oriented edges normal to $Q$ at $p$ (resp. $q$), i.e. ${\dsp\left(E_c\right)^{pq}\setminus\{Q\}\cong N^p_Q \left(\text{resp.} N^q_Q\right)}$.  Similarly identify the oriented $c$-edges at $t(Q)=vw$ with the oriented edges normal to $Q$ at $v$ (resp. $w$), i.e.  ${\dsp\left(E_c\right)_{vw}\setminus\{\bar{Q}\}\cong N^v_Q \left(\text{resp.} N^w_Q\right)}$.  We have the following diagram:
\[\xymatrixcolsep{5pc}\xymatrixrowsep{3pc}\xymatrix{(E_c)_{\ol{pq}}\setminus\{Q\}\ar[r]^-{(\theta_c^u)_Q}\ar[d]_-\cong & (E_c)_{vw}\setminus\{\bar{Q}\}\ar[d]_-\cong\\
N_Q^q\ar[r]^-{K^\perp_{\gamma_Q^u}}\ar[d]_-{\theta^\perp_{qp}} & N_Q^w\ar[d]^-{\theta^\perp_{vw}}\\
N_Q^p\ar[r]^-{K^\perp_{\gamma_Q^d}}\ar[d]_-\cong & N_Q^v\ar[d]^-\cong\\
(E_c)_{pq}\setminus\{Q\}\ar[r]^-{(\theta_c^d)_Q} & (E_c)_{vw}\setminus\{\bar{Q}\}.\\}\]
The vertical maps compose to the identity maps on ${\dsp\left(E_c\right)_{pq}\setminus\{Q\}}$ and ${\dsp\left(E_c\right)_{vw}\setminus\{Q\}}$.  The top and bottom squares commute by definition, i.e. Eq. \eqref{eq:upconnection}, and the commutativity of the middle square follows from the trivial normal holonomy of $Q$.  Thus the outer rectangle commutes, which implies the equivalence of $\theta_c^u$ and $\theta_{c}^d$.  

Using Eqs. \eqref{eq:upaxial} and \eqref{eq:downaxial}, and that 
$$\alpha(qr_2)-\lambda_{qp}(qr_2)\alpha(pt_2)=c\alpha(qp)=-c\alpha(pq),$$
we see 
$$\alpha_c^u(Q)=\kappaup(Q)\cdot\alpha_{c'}^d(Q) \ \ \ \text{where} \ \ \ \kappaup(Q)\coloneqq\lambda_{qp}(qr_2).$$
The trivial normal holonomy of $Q$ yields 
$$K^\perp_{\gamma_Q^d}\circ\theta^\perp_{qp}=\theta^\perp_{wv}\circ K^\perp_{\gamma_Q^u},$$
which implies that
$$\theta^d_c=\theta^u_c.$$
Let $R\in(E_c)_{{pq}}\setminus\{Q\}$ be an oriented $c$-edge at $i(Q)=pq$ corresponding to oriented edges ${qu_1}\in N^0_q$ and ${ps_1}\in N^0_p$.  Then the levelness of $Q$ implies
$$\left|K^\perp_{\gamma_Q^d}(ps_1)\right|\cdot\lambda_{qp}(qu_1)=\lambda_{wv}\left(K^\perp_{\gamma^u_Q}(qu_1)\right)\cdot\left|K^\perp_{\gamma_Q^u}(qu_1)\right|,$$ 
from which it follows that
$$\left(\lambda_c^d\right)_Q(R)\cdot\frac{\kappaup(R)}{\kappaup\left(\left(\theta_c^u\right)_Q(R)\right)}=\left(\lambda_c^u\right)_Q(R).$$

Now we establish (ii).  There is a natural identification of the vertices which extends to an identification of graphs $\Gamma_c^\sharp$ and $\Gamma_{c'}^\sharp$:
\begin{equation}
\label{eq:updownid}
\xymap{V_c\setminus V_{c,0}\ar@{=}[r] & V_{c'}\setminus V_{c',0}\\
\sqcup & \sqcup\\
\left\{Z_{ia}\right\}\ar[r] & \left\{Z_{ai}\right\}.\\}
\end{equation}
Note that $c$-vertices outside the singular locus of $\Gamma_c^\sharp$ coincide with the $c'$-vertices outside the singular locus of $\Gamma_{c'}^\sharp$.  Indeed oriented edges not containing $p$ are at $c$-level if and only if they are at $c'$-level.  Also note that Eq. \eqref{eq:updownid} identifies horizontal edges of $\Gamma_{c,0}^\sharp$ with vertical edges of $\Gamma_{c',0}^\sharp$ and vice versa.  In order to establish (ii), we need to show 
\begin{enumerate}[(a)]
\item ${\dsp \left(\theta_c^u\right)^\sharp_\epsilon=\left(\theta_{c'}^d\right)^\sharp_\epsilon}$
{}\\
\item ${\dsp \left(\lambda_c^u\right)^\sharp_\epsilon=\frac{\kappaup}{\kappaup\circ\left(\theta_c^u\right)^\sharp_\epsilon}\left(\lambda_{c'}^d\right)^\sharp_\epsilon}$
{}\\
\item ${\dsp \left(\alpha_c^u\right)^\sharp=\kappaup\left(\alpha^d_{c'}\right)^\sharp}$.
\end{enumerate}
Note that for oriented edges not issuing from the singular locus of either $\Gamma_c^\sharp$ or $\Gamma_{c'}^\sharp$, the blow-up connection maps, compatibility systems, and generalized axial functions coincide with their counter parts on $\Gamma_c$ or $\Gamma_{c'}$, cf. Eqs. \eqref{eq:thetanotlocus}, \eqref{eq:lambdanotlocus}, and \eqref{eq:alphanotlocus}.
Therefore by (i), it suffices to establish (a), (b), and (c) for oriented edges $\epsilon$ issuing from the singular locus of $\Gamma_c^\sharp \cong \Gamma_{c'}^\sharp$.

A direct comparison of values in Table \ref{tab:buequivthetavalues} shows that $\left(\theta_c^u\right)^\sharp_\epsilon=\left(\theta_{c'}^d\right)^\sharp_\epsilon$, hence (a) holds.

\begin{table}
\caption{$\left(\theta^u_{c^{}}\right)^\sharp=\left(\theta^d_{c^{\prime}}\right)^\sharp$}
\label{tab:buequivthetavalues}  
\begin{tabular}{llll}
\hline\noalign{\smallskip}
$\epsilon$ & $\left(\theta^u_{c^{}}\right)^\sharp_\epsilon\left(Z_{ia}[t(Q_{ia})]\right)$ & $\left(\theta^u_{c^{}}\right)^\sharp_\epsilon\left(Z_{ia}Z_{ja}\right)$ & $\left(\theta^u_{c^{}}\right)^\sharp_\epsilon\left(Z_{ia}Z_{ib}\right)$ \\
\noalign{\smallskip}\hline\noalign{\smallskip}
$Z_{ia}[t(Q_{ia})]$ & --------- & $\left(\theta_c^u\right)_{Q_{ia}}\left(Q_{ik}\right)$ & $\left(\theta_c^u\right)_{Q_{ia}}\left(Q_{ic}\right)$ \\
\noalign{\smallskip}
$Z_{ia}Z_{ja}$ & $Z_{ja}[t(Q_{ja})]$ & $Z_{ja}Z_{ka}$ & $Z_{ja}Z_{jc}$\\ 
\noalign{\smallskip}
$Z_{ia}Z_{ib}$ & $Z_{ib}[t\left(Q_{ib}\right)]$ & $Z_{ib}Z_{kb}$ & $Z_{ib}Z_{ic}$\\
\noalign{\smallskip}\hline
\end{tabular}
\begin{tabular}{llll}
\hline\noalign{\smallskip}
$\epsilon$ & $\left(\theta^d_{c^{\prime}}\right)^\sharp_\epsilon\left(Z_{ai}[t(Q_{ai})]\right)$ & $\left(\theta^d_{c^{\prime}}\right)^\sharp_\epsilon\left(Z_{ai}Z_{aj}\right)$ & $\left(\theta^d_{c^{\prime}}\right)^\sharp_\epsilon\left(Z_{ai}Z_{bi}\right)$ \\
\noalign{\smallskip}\hline\noalign{\smallskip}
$Z_{ai}[t(Q_{ai})]$ & --------- & $\left(\theta_{c^\prime}^d\right)_{Q_{ai}}\left(Q_{ak}\right)$ & $\left(\theta_{c^\prime}^d\right)_{Q_{ai}}\left(Q_{ci}\right)$ \\
\noalign{\smallskip}
$Z_{ai}Z_{aj}$ & $Z_{aj}[t(Q_{aj})]$ & $Z_{aj}Z_{ak}$ & $Z_{aj}Z_{cj}$\\ 
\noalign{\smallskip}
$Z_{ai}Z_{bi}$ & $Z_{bi}[t\left(Q_{bi}\right)]$ & $Z_{bi}Z_{bk}$ & $Z_{bi}Z_{ci}$\\
\noalign{\smallskip}\hline
\end{tabular}
\end{table}

Tables \ref{tab:buequivlambdavalues} and \ref{tab:buequivalphavalues} compare the compatibility systems and generalized axial functions on $\Gamma_c^\sharp$ and $\Gamma_{c'}^\sharp$.  In Table~\ref{tab:buequivalphavalues} the $c$-edges $Z_{ia}[t(Q_{ia}]$, $Z_{ia}Z_{ja}$, and $Z_{ia}Z_{ib}$ are implicitly identified with the $c'$-edges $Z_{ai}[t(Q_{ai}]$, $Z_{ai}Z_{aj}$, and $Z_{ai}Z_{bi}$, respectively.
\begin{table}
\caption{$\left(\lambda_c^u\right)^\sharp_{\epsilon}(\epsilon^\prime)\equiv\left(\lambda_{c'}^d\right)^\sharp_\epsilon(\epsilon^\prime)$}
\label{tab:buequivlambdavalues}  
\begin{tabular}{llll}
\hline\noalign{\smallskip}
$\epsilon$ & $\left(\lambda^u_{c^{}}\right)^\sharp_\epsilon\left(Z_{ia}[t(Q_{ia})]\right)$ & $\left(\lambda^u_{c^{}}\right)^\sharp_\epsilon\left(Z_{ia}Z_{ja}\right)$ & $\left(\lambda^u_{c^{}}\right)^\sharp_\epsilon\left(Z_{ia}Z_{ib}\right)$ \\
\noalign{\smallskip}\hline\noalign{\smallskip}
$Z_{ia}[t(Q_{ia})]$ & -- & $\left(\lambda_c^u\right)_{Q_{ia}}\left(Q_{ik}\right)$ & $\left(\lambda_c^u\right)_{Q_{ia}}\left(Q_{ic}\right)$ \\
\noalign{\smallskip}
$Z_{ia}Z_{ja}$ & $1$ & $\left(\lambda_c^u\right)_{Q_{ij}}\left(Q_{ik}\right)$ & $\left(\lambda_c^u\right)_{Q_{ij}}\left(Q_{ic}\right)$\\ 
\noalign{\smallskip}
$Z_{ia}Z_{ib}$ & $1$ & $1$ & $1$\\
\noalign{\smallskip}\hline
\end{tabular}
\begin{tabular}{llll}
\hline\noalign{\smallskip}
$\epsilon$ & $\left(\lambda^d_{c^{\prime}}\right)^\sharp_\epsilon\left(Z_{ai}[t(Q_{ai})]\right)$ & $\left(\lambda^d_{c^{\prime}}\right)^\sharp_\epsilon\left(Z_{ai}Z_{aj}\right)$ & $\left(\lambda^d_{c^{\prime}}\right)^\sharp_\epsilon\left(Z_{ai}Z_{bi}\right)$ \\
\noalign{\smallskip}\hline\noalign{\smallskip}
$Z_{ai}[t(Q_{ai})]$ & -- & $\left(\lambda_{c^\prime}^d\right)_{Q_{ai}}\left(Q_{ak}\right)$ & $\left(\lambda_{c^\prime}^d\right)_{Q_{ai}}\left(Q_{ci}\right)$ \\
\noalign{\smallskip}
$Z_{ai}Z_{aj}$ & $1$ & $1$ & $1$\\ 
\noalign{\smallskip}
$Z_{ai}Z_{bi}$ & $1$ & $\left(\lambda_{c^\prime}^d\right)_{Q_{ab}}\left(Q_{ak}\right)$ & $\left(\lambda_{c^\prime}^d\right)_{Q_{ab}}\left(Q_{ac}\right)$\\
\noalign{\smallskip}\hline
\end{tabular}
\end{table}

\begin{table}
\caption{$\left(\alpha_c^u\right)^\sharp\equiv \left(\alpha_{c'}^d\right)^\sharp$}
\label{tab:buequivalphavalues}       
\begin{tabular}{lll}
\hline\noalign{\smallskip}
$\epsilon$ & $\left(\alpha_c^u\right)^\sharp(\epsilon)$ & $\left(\alpha_{c'}^d\right)^\sharp(\epsilon)$\\
\noalign{\smallskip}\hline\noalign{\smallskip}
$Z_{ia}[t(Q_{ia})]$ & ${\dsp \frac{1}{n^u\left(Q_{ia}\right)}\alpha_c^u\left(Q_{ia}\right)}$ & ${\dsp \frac{1}{n^d\left(Q_{ai}\right)}\alpha_c^u\left(Q_{ai}\right)}$\\
$Z_{ia}Z_{ja}$ & $\alpha_c^u\left(Q_{ij}\right)$ & ${\dsp\alpha_{c'}^d\left(Q_{aj}\right)- \frac{n^d\left(Q_{aj}\right)}{n^d\left(Q_{ai}\right)}\alpha_{c'}^d\left(Q_{ai}\right)}$ \\
$Z_{ia}Z_{ib}$ & ${\dsp \alpha_c^u\left(Q_{ib}\right)-\frac{n\left(Q_{ib}\right)}{n\left(Q_{ia}\right)}\alpha_c^u\left(Q_{ia}\right)}$ & $\alpha_{c'}^d\left(Q_{ab}\right)$\\
\noalign{\smallskip}\hline
\end{tabular}
\end{table}

Recalling that 
${\dsp \begin{cases}
\alpha_c^u\left(Q_{ij}\right)= & \alpha\left(pp_j\right)-\frac{\langle \alpha\left(pp_j\right),\xi\rangle}{\langle\alpha\left(pp_i\right),\xi\rangle}\alpha\left(pp_i\right)\\
{}\\
\alpha_c^u\left(Q_{ia}\right)= & \alpha\left(pq_a\right)-\frac{\langle \alpha\left(pq_a\right),\xi\rangle}{\langle\alpha\left(pp_i\right),\xi\rangle}\alpha\left(pp_i\right)\\
{}\\
\alpha_{c'}^d\left(Q_{ai}\right)= & \alpha\left(pp_i\right)-\frac{\langle\alpha\left(pp_i\right),\xi\rangle}{\langle\alpha\left(pq_a\right),\xi\rangle}\alpha\left(pq_a\right)\\
{}\\
\alpha_{c'}^d\left(Q_{ab}\right)= & \alpha\left(pq_b\right)-\frac{\langle\alpha\left(pq_b\right),\xi\rangle}{\langle\alpha\left(pq_a\right),\xi\rangle}\alpha\left(pq_a\right)\\
\end{cases}}$

and that 
${\dsp \begin{cases}
n^u\left(Q_{ia}\right)= & \langle \alpha\left(pq_a\right),\xi\rangle\\
n^d\left(Q_{ai}\right)= & \langle \alpha\left(p_ip\right),\xi\rangle,\\
\end{cases}}$

it is straight forward to check that 
$$\begin{array}{lccl}\frac{1}{n^u\left(Q_{ia}\right)}\alpha_c^u\left(Q_{ia}\right) & = & 1 & \cdot \left(\frac{1}{n^d\left(Q_{ai}\right)} \alpha_{c'}^d\left(Q_{ai}\right)\right)\\
\alpha_c^u\left(Q_{ij}\right) & = & \frac{1}{\langle\alpha(p_ip),\xi\rangle} &   \cdot\left(\alpha_{c'}^d\left(Q_{aj}\right)-\frac{n^d\left(Q_{aj}\right)}{n^d\left(Q_{ai}\right)}\alpha_{c'}^d\left(Q_{ai}\right)\right)\\
\alpha_{c}^u\left(Q_{ib}\right)-\frac{n^u\left(Q_{ib}\right)}{n^u\left(Q_{ia}\right)}\alpha_{c}^u\left(Q_{ia}\right) & = & \langle\alpha(pq_a),\xi\rangle & \cdot\left(\alpha_{c'}^d\left(Q_{ab}\right)\right).
\end{array}$$

Thus we deduce that the values of $\kappaup(\epsilon)$ that make (c) hold are as shown in Table \ref{tab:kappavalues}.  The verification that (b) also holds with $\kappaup$ defined by Table \ref{tab:kappavalues} is straight forward, and is left to the reader.  This establishes the equivalence in (ii), and thereby completes the proof of Lemma \ref{lem:equivblowup}.

\begin{table}
\caption{$\kappaup$ values}
\label{tab:kappavalues}       
\begin{tabular}{lll}
\hline\noalign{\smallskip}
$\epsilon$ & $\kappaup(\epsilon)$\\
\noalign{\smallskip}\hline\noalign{\smallskip}
$Q$ & $\lambda_{qp}(qr_2)$ \\
$Z_{ia}[t(Q_{ia})]$ & $1$\\
$Z_{ia}Z_{ja}$ & $\frac{1}{\langle\alpha(p_ip),\xi\rangle}$\\
$Z_{ia}Z_{ib}$ & $\langle\alpha(pq_a),\xi\rangle$\\
\noalign{\smallskip}\hline
\end{tabular}
\end{table}
\end{proof}

\subsection{Cutting}
Let $(\Gamma,\alpha,\theta,\lambda)$ be a $d$-valent reducible 1-skeleton in $\R^n$, $\xi$ a polarizing covector in $(\R^n)^*$, and $\phi$ a compatible Morse function.
Let ${\dsp I=\left(\left\{0,1\right\},\left\{01,10\right\}\right)}$ denote the interval graph (i.e. connected single edge graph), and let $\theta_I$
denote the unique connection on $I$.  Then $\lambda_I\equiv 1$ defines a compatibility system on $(I,\theta_I)$ and the function ${\dsp\alpha_I\colon 
\begin{cases}
01\mapsto 1\\
10\mapsto -1\\
\end{cases}}$ defines an axial function, making the quadruple $(I,\alpha_I,\theta_I,\lambda_I)$ a 1-skeleton in $\R$.

The \emph{direct product} 1-skeleton ${\dsp(\hat{\Gamma},\hat{\alpha},\hat{\theta},\hat{\lambda})\subset\R^n\times\R}$ with factors  $(\Gamma,\alpha,\theta,\lambda)\subset\R^n$ and $(I,\alpha_I,\theta_I,\lambda_I)\subset\R$ is a $(d+1)$-valent 1-skeleton defined as follows.  The graph $\hat{\Gamma}$ has vertex set ${\dsp V_{\hat{\Gamma}}\coloneqq V_\Gamma\times 0\sqcup V_\Gamma\times 1}$ and oriented edge set ${\dsp E_{\hat{\Gamma}}\coloneqq E_\Gamma\times V_I \sqcup V_\Gamma\times E_I}$.  The connection maps are defined as the product of the connection maps of the factors:  ${\dsp \hat{\theta}_\epsilon\coloneqq \left(\theta\times\theta_I\right)_\epsilon\colon E_{\hat{\Gamma}}^{i(\epsilon)}\rightarrow E_{\hat{\Gamma}}^{t(\epsilon)}.}$  The compatibility system maps are likewise the product of the compatibility system maps of the factors:  ${\dsp \hat{\lambda}_\epsilon\coloneqq \left(\lambda\times\lambda_I\right)_\epsilon\colon E_{\hat{\Gamma}}^{i(\epsilon)}\rightarrow \R_+.}$  The product axial function is ${\dsp \xymap{E_{\hat{\Gamma}}\ar[r]^-{\hat{\alpha}} & \R^n\times\R\\
e\times v\ar@{|->}[r] & \left(\alpha(e),0\right)\\
v\times e\ar@{|->}[r] & \left(0,1\right).\\}}$ 

Let $\mathbf{1}\in\R^*$ denote the linear functional on $\R$ that maps $1$ to $1$.  Define the covector
${\dsp\hat{\xi}\coloneqq\left(\xi,\mathbf{1}\right)\in(\R^n\times\R)^* \cong(\R^n)^*\times\R^*.}$
Then $\hat{\xi}$ is generic and polarizing for $(\hat{\Gamma},\hat{\alpha}_\eta,\hat{\theta},\hat{\lambda})$ since
\begin{align*}
\langle\hat{\xi},\hat{\alpha}(e\times v)\rangle& =\langle\xi,\alpha(e)\rangle\\
\langle\hat{\xi},\hat{\alpha}(v\times e)\rangle & =\langle\mathbf{1},\alpha_I(e)\rangle.
\end{align*}
Also note that $(\hat{\Gamma},\hat{\alpha},\hat{\theta},\hat{\lambda})$ has enough $2$-faces; $2$-faces are either $Q\times v$ for $Q$ a $2$-face of $(\Gamma,\alpha,\theta,\lambda)$ or rectangles $e\times e'$.  Thus $(\hat{\Gamma},\hat{\alpha},\hat{\theta},\hat{\lambda})$ is reducible.

Set $\phi_-\coloneqq\min_{p\in V_\Gamma}(\phi(p))$ and $\phi_+\coloneqq\max_{p\in V_\Gamma}(\phi(p))$, and fix $a>\phi_+-\phi_->0$.  Define a $\hat{\xi}$-compatible Morse function by 
$$\xymap{ V_{\hat{\Gamma}}\ar[r]^-{\hat{\phi}} & \R\\
v\times t\ar@{|->}[r] &  \phi(v)+at.}$$

\begin{lemma}
\label{lem:cutlift}
In the notation above, if $(\Gamma,\alpha,\theta,\lambda)$ satisfies $\left(\dagger\right)$ in Theorem \ref{thm:ch2mainresult}, then the direct product 1-skeleton $(\hat{\Gamma},\hat{\alpha},\hat{\theta},\hat{\lambda})$ also satisfies $\left(\dagger\right)$.  Moreover if $c\in\R$ is any $\hat{\phi}$-regular value such that
$$\phi_+<c<\phi_-+a$$
then there is a linear isomorphism $\pi\colon W_{\hat{\xi}}\rightarrow\R^n$ that yields an equivalence of 1-skeleta $(\hat{\Gamma}_c,\pi\circ\hat{\alpha}_c^d,\hat{\theta}_c^d,\hat{\lambda}_c^d)\equiv(\Gamma,\alpha,\theta,\lambda)$.
\end{lemma}
\begin{proof}
Using the notation above, assume that $(\Gamma,\alpha,\theta,\lambda)$ satisfies $(\dagger)$.  The direct product has two types of $2$-faces:
\begin{enumerate}[(i)]
\item $2$-faces of the form $Q\times v$, for $Q$ a $2$-face of $(\Gamma,\alpha,\theta,\lambda)$.
\item Rectangles of the form $e\times e'$, for $e\in E_\Gamma$ and $e'\in E_I$.
\end{enumerate}
Clearly $2$-faces of type (i) are level with trivial normal holonomy since the $\hat{\theta}$ agrees with $\theta$ and $\hat{\lambda}$ agrees with $\lambda$ on $Q\times v$.  Let $F\coloneqq pq\times 01$ be a $2$-face of type (ii), and let $\gamma_F$ denote the loop $p\times 0\sra q\times 0\sra q\times 1\sra p\times 1\sra p\times 0$.  Then the normal holonomy map $K_{\gamma_F}^\perp$ has the form 
$$K_{\gamma_F}^\perp=\left(\theta_I\right)_{10}\circ\theta_{qp}\circ\left(\theta_I\right)_{01}\circ\theta_{pq},$$
which is clearly the identity map on $N_F^{p\times 0}$.  Moreover the normal holonomy number has the form
$$\left|K_{\gamma_F}^\perp\right|=\left(\lambda_I\right)_{10}\cdot\lambda_{qp}\cdot\left(\lambda_I\right)_{01}\cdot\lambda_{pq},$$
which is identically one on $N_F^{p\times 0}$.  This shows that $(\dagger)$ holds for $(\hat{\Gamma},\hat{\alpha},\hat{\theta},\hat{\lambda})$.

To establish the equivalence of $(\Gamma,\alpha,\theta,\lambda)$ with the down $c$-cross section of $(\hat{\Gamma},\hat{\alpha},\hat{\theta},\hat{\lambda})$, first note that the oriented edges at $c$-level are those of the form $v\times 01$.  Indeed all vertices of the form $p\times 0$ have ${\dsp \hat{\phi}(p\times 0)=\phi(p)\leq\phi_+ \ < \ c}$ and all vertices of the form $q\times 1$ have ${\dsp \hat{\phi}(q\times 1)=\phi(q)+a\geq\phi_- +a \ > \ c.}$  Similarly the oriented $2$-faces at $c$-level are the oriented rectangles of the form $e\times 01$.  This defines a bijection between graphs $\Gamma$ and $\hat{\Gamma}_c$:
$$\xymap{V_\Gamma \sqcup E_\Gamma \ar[r] & \hat{V}_c\sqcup \hat{E}_c\\
v\ar@{|->}[r] & v\times 01\\
e\ar@{|->}[r] & e\times 01.\\}$$
The down connection map ${\dsp \left(\hat{\theta}_c^d\right)_{e\times 01}=K_{\gamma^{e}}^\perp}$ is simply equal to $\theta_e$ by definition.  Likewise, the down compatibility map ${\dsp \left(\hat{\lambda}_c^d\right)_{e\times 01}=\left|K_{\gamma^{e}}^\perp\right|}$ is equal to $\lambda_e$.  Finally we have
\begin{align*}
\hat{\alpha}_c^d\left(pq\times 01\right) & =\hat{\alpha}(pq\times 0)-\frac{\langle\hat{\xi},\hat{\alpha}(pq\times 0)\rangle}{\langle\hat{\xi},\hat{\alpha}(p\times 01)\rangle}\hat{\alpha}(p\times 01)\\
{}\\
& =\left(\begin{array}{c}\alpha(pq)\\ 0\\ \end{array}\right)-\frac{\langle\xi,\alpha(pq)\rangle}{1}\left(\begin{array}{c} 0\\ 1\\ \end{array}\right)\\
{}\\
& =\left(\begin{array}{c} \alpha(pq)\\ -\langle\xi,\alpha(pq)\rangle \\ \end{array}\right).
\end{align*}
Thus the restriction of the projection map $\pi\colon\R^n\times\R\rightarrow\R^n$ to the vanishing hyperplane of $\hat{\xi}$, $W_{\hat{\xi}}$, is a linear isomorphism such that $\pi\circ\hat{\alpha}_c^d=\alpha$.
\end{proof}

\begin{figure}
\includegraphics{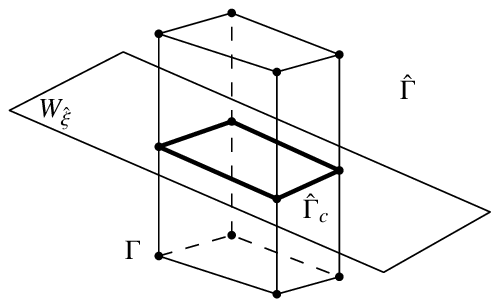}
\caption{Cutting}
\label{fig:cuttingill}
\end{figure}

\section{Application:  Projecting and Lifting 1-Skeleta}
\subsection{Projections of 1-skeleta}
Fix a $d$-valent pre 1-skeleton $(\Gamma,\theta,\lambda)$, let $A\colon E_\Gamma\rightarrow\R^N$ be an effective generalized axial function and let $p\colon\R^N\rightarrow \R^n$ be any surjective linear map.  Then $p\circ A$, is also an effective generalized axial function for $(\Gamma,\theta,\lambda)$, and the resulting generalized 1-skeleton, $(\Gamma,p\circ A,\theta,\lambda)$, is called the \emph{projection} of $(\Gamma,A,\theta,\lambda)$.  

Conversely, we say a generalized 1-skeleton $(\Gamma,\alpha,\theta,\lambda)$ in $\R^n$ has a \emph{lift} if there is an effective generalized 1-skeleton $(\Gamma,A,\theta,\lambda)$ in $\R^N$ ($n \ < \ N \ \leq \ d$) and a surjective linear map $p\colon\R^N\rightarrow\R^n$ such that $\alpha=p\circ A$. We say $(\Gamma,\alpha,\theta,\lambda)$ has a \emph{total lift} if it has a lift to an effective generalized 1-skeleton $(\Gamma,A,\theta,\lambda)$ in $\R^d$.  Note that an effective axial function $A\colon E_\Gamma\rightarrow\R^d$ for a $d$-valent pre 1-skeleton $(\Gamma,\theta,\lambda)$ is necessarily $d$-independent.  

\begin{figure*}
  \includegraphics[width=0.75\textwidth]{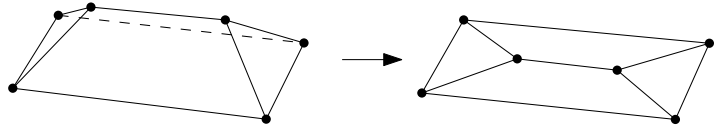}
\caption{Projection}
\label{fig:Projection}       
\end{figure*}

For the convenience of the reader, we restate Theorem \ref{thm:ch2mainresult} here.
\begin{theorem1*}
Let $(\Gamma,\alpha,\theta,\lambda)\subset\R^n$ be a $d$-valent reducible 1-skeleton in $\R^n$.  Then $(\Gamma,\alpha,\theta,\lambda)$ has a total lift if and only if\\
${\dsp\left(\dagger\right) \ \ \  \text{Every $2$-face of $(\Gamma,\alpha,\theta,\lambda)$ is level and has trivial normal holonomy}.}$
\end{theorem1*}
Note that a (total) lift $(\Gamma,A,\theta,\lambda)$ of a reducible 1-skeleton $(\Gamma,\alpha,\theta,\lambda)$ is necessarily reducible.  Indeed any generic covector $\xi\in\left(\R^n\right)^*$ for $(\Gamma,\alpha,\theta,\lambda)$ pulls back via the projection $p\colon\R^d\rightarrow\R^n$ to a generic covector $\Xi=\xi\circ p\in\left(\R^d\right)^*$ for $(\Gamma,A,\theta,\lambda)$ satisfying $\langle\Xi,A(e)\rangle=\langle\xi,\alpha(e)\rangle$.  Since $(\Gamma,A,\theta,\lambda)$ and $(\Gamma,\alpha,\theta,\lambda)$ have the same connection and polarization, they must also have the same $2$-faces.
\subsection{Proof of Theorem \ref{thm:ch2mainresult}}
Before proving Theorem \ref{thm:ch2mainresult} we will need the following lemmata.  The first lemma asserts that total liftability is preserved by a blow-up or a blow-down.  

Fix a $d$-valent generalized 1-skeleton $(\Gamma,\alpha,\theta,\lambda)$ in $\R^n$, a $k$-valent (level) subskeleton $(\Gamma_0,\alpha_0,\theta_0,\lambda_0)$, and a blow-up system for the subskeleton $n\colon N_0\rightarrow\R_+$.  Let $(\Gamma^\sharp,\alpha^\sharp,\theta^\sharp,\lambda^\sharp)$ denote the corresponding blow-up generalized 1-skeleton.

\begin{lemma}
\label{lem:blowupcommuteslift}
$(\Gamma,\alpha,\theta,\lambda)$ has a total lift if and only if $(\Gamma^\sharp,\alpha^\sharp,\theta^\sharp,\lambda^\sharp)$ has a total lift.  
\end{lemma}
\begin{proof}
Let $(\Gamma,A,\theta,\lambda)$ be a total lift of $(\Gamma,\alpha,\theta,\lambda)$ via a surjective map $p\colon\R^d\rightarrow\R^n$.  Then the function $A^\sharp\colon E^\sharp\rightarrow\R^d$ defined as in Definition \ref{def:blowup1sk} is an axial function for $(\Gamma^\sharp,\theta^\sharp,\lambda^\sharp)$.  Moreover the linearity of $p$ guarantees that $p\circ A^\sharp=\alpha^\sharp$.  Finally note that ${\dsp \left\{A^\sharp(e)\left|\right. e\in E^p\right\}}$ and ${\dsp\left\{A(e)\left|\right. e\in E^p\right\}}$ coincide for $p\notin V_0^\sharp$.  Thus since $A$ is $d$-independent, $A^\sharp$ must also be $d$-independent.  

Conversely let $(\Gamma^\sharp,\tilde{A},\theta^\sharp,\lambda^\sharp)$ be a total lift of $(\Gamma^\sharp,\alpha^\sharp,\theta^\sharp,\lambda^\sharp)$ via a surjective map $p\colon\R^d\rightarrow\R^n$.  In order to define a total lift $A\colon E_\Gamma\rightarrow\R^d$ for $\alpha$, it suffices to see that $\tilde{A}$ is constant on the fibers of the blow-up morphism.  Indeed if $A$ is constant on the fibers $\beta^{-1}(e)$ for each $e\in E_0$, then the function $A\colon E_\Gamma\rightarrow\R^d$ given by
\begin{equation}
\label{eq:blaxials}
A(e)=\begin{cases}
n(e)\tilde{A}(z^p_e[t(e)]) & \text{if $e\in N_0$}\\
\tilde{A}(\epsilon) & \text{if $e\in E_\Gamma\setminus N_0$ and $\epsilon\in\beta^{-1}(e)$}\\
\end{cases}
\end{equation}
is well-defined.  Fix a vertex $p\in V_\Gamma$, oriented edges $e,e'\in E^p$, and consider the difference 
\begin{equation}
\label{eq:ADifference}
A(e')-\lambda_e(e')A(\theta_e(e')).
\end{equation}
Fix $\epsilon\in \beta^{-1}(e)$ and $\epsilon'\in \beta^{-1}(e')$.  For $e'\in N_0$, Eq. \eqref{eq:ADifference} becomes 
$$n(e')\tilde{A}(\epsilon')-\lambda_e(e')\cdot n(\theta_e(e'))\tilde{A}(\theta^\sharp_{\epsilon}(\epsilon'))=n(e')\left(\tilde{A}(\epsilon')- 1\cdot\tilde{A}(\theta^\sharp_{\epsilon}(\epsilon'))\right),$$
which is a multiple of $\tilde{A}(\epsilon)=\frac{1}{n(e)}A(e)$.
For $e'\notin N_0$, Eq. \eqref{eq:ADifference} becomes
$$\tilde{A}(\epsilon')-\lambda^\sharp_{\epsilon}(\epsilon')\tilde{A}(\theta^\sharp_\epsilon(\epsilon')),$$
which is also a multiple of $\tilde{A}(\epsilon)=A(e)$.  Thus Eq. \eqref{eq:blaxials} defines a generalized axial function for $(\Gamma,\theta,\lambda)$.  Moreover it is clear that $A$ is $d$-independent (since $\tilde{A}$ is $d$-independent), and that $p\circ A=\alpha$.  Thus it remains to see that $\tilde{A}$ is constant on $\beta^{-1}(e)$ for each $e\in E_0$.

Fix $e\in E_0$ and let $\epsilon',\epsilon''$ be any two edges in the fiber $\beta^{-1}(e)$.  We need to show that $\tilde{A}(\epsilon')=\tilde{A}(\epsilon'')$.  If $\epsilon'$ and $\epsilon''$ are not distinct there is nothing to show.  Otherwise $\epsilon'$ and $\epsilon''$ are distinct edges in the same fiber, hence there must be a vertical edge $\epsilon$ joining the vertices $i(\epsilon')$ and $i(\epsilon'')$.  Thus we have 
\begin{equation}
\label{eq:liftblow}
\tilde{A}(\epsilon')-\lambda^\sharp_{\epsilon}(\epsilon')\tilde{A}(\epsilon'')=c\tilde{A}(\epsilon)
\end{equation}
for some $c\in\R$.  The claim is that $c=0$.  To see this, apply the projection $p$ to both sides of \eqref{eq:liftblow} to get
\begin{equation}
\label{eq:liftblowp}
\alpha^\sharp(\epsilon')-\lambda^\sharp_{\epsilon}(\epsilon')\alpha^\sharp(\epsilon'')=c\alpha^\sharp(\epsilon);
\end{equation}
Since $\lambda^\sharp_{\epsilon}(\epsilon')=1$ and $\alpha^\sharp(\epsilon')=\alpha^\sharp(\epsilon'')$, we deduce that the LHS of \eqref{eq:liftblowp} must be zero.  This implies that $c$ must also be equal to zero, as desired.
\end{proof}

The next lemma asserts that any down cross section of a reducible 1-skeleton satisfying $\left(\dagger\right)$ must have a total lift.  See Fig.~\ref{fig:PrettyPassage}.

\begin{lemma}
\label{lem:downcrosslift}
Let $(\Gamma,\alpha,\theta,\lambda)$ be a $d$-valent reducible 1-skeleton satisfying $(\dagger)$ in Theorem \ref{thm:ch2mainresult}, $\xi$ a polarizing covector for $(\Gamma,\alpha,\theta,\lambda)$, and $\phi$ a compatible Morse function.  Then for any $\phi$-regular value $c\in\R$, the $(d-1)$-valent generalized 1-skeleton $(\Gamma_c,\alpha_c^d,\theta_c^d,\lambda_c^d)$ has a total lift.
\end{lemma}
\begin{proof}
Let $c_0<c_1<\cdots<c_N<c_{N+1}$ be $\phi$-regular values such that for each $p\in V_\Gamma$, $c_0<\phi(p)< c_{N+1}$ and for each $0\leq i\leq N$ there is a unique vertex $p_i\in V_\Gamma$ satisfying $c_i<\phi(p_i)<c_{i+1}$.  To complete the proof it suffices to show that $(\Gamma_{c_1},\alpha_{c_1}^d,\theta_{c_1}^d,\lambda_{c_1}^d)$ has a total lift.  Indeed inductively assume that $(\Gamma_{c_i},\alpha_{c_i}^d,\theta_{c_i}^d,\lambda_{c_i}^d)$ has a total lift.  Since all $2$-faces are level with trivial normal holonomy, Lemma \ref{lem:equivblowup} implies that
${\dsp(\Gamma_{c_i},\alpha_{c_i}^u,\theta_{c_i}^u,\lambda_{c_i}^u)\equiv (\Gamma_{c_i},\alpha_{c_i}^d,\theta_{c_i}^d,\lambda_{c_i}^d).}$
Hence $(\Gamma_{c_i},\alpha_{c_i}^u,\theta_{c_i}^u,\lambda_{c_i}^u)$ must also have a total lift.  By Lemma \ref{lem:blowupcommuteslift}, the blow-up
${\dsp(\Gamma_{c_i}^\sharp,(\alpha_{c_i}^u)^\sharp,(\theta_{c_i}^u)^\sharp,(\lambda_{c_i}^u)^\sharp)}$
has a total lift.  Again Lemma \ref{lem:equivblowup} implies that $(\Gamma_{c_i}^\sharp,(\alpha_{c_i}^u)^\sharp,(\theta_{c_i}^u)^\sharp,(\lambda_{c_i}^u)^\sharp)\equiv(\Gamma_{c_{i+1}}^\sharp,(\alpha_{c_{i+1}}^d)^\sharp,(\theta_{c_{i+1}}^d)^\sharp,(\lambda_{c_{i+1}}^d)^\sharp)$.
Hence ${\dsp(\Gamma_{c_{i+1}}^\sharp,(\alpha_{c_{i+1}}^d)^\sharp,(\theta_{c_{i+1}}^d)^\sharp,(\lambda_{c_{i+1}}^d)^\sharp)}$ has a total lift as well.  Hence by Lemma \ref{lem:blowupcommuteslift} again,  $(\Gamma_{c_{i+1}},\alpha_{c_{i+1}}^d,\theta_{c_{i+1}}^d,\lambda_{c_{i+1}}^d)$ must also have a lift, and so on.

Note that the graph of the $c_1$-cross section ${\dsp (\Gamma_{c_1},\alpha_{c_1}^d,\theta_{c_1}^d,\lambda_{c_1}^d)}$ is a complete graph on $d$ vertices.  For concreteness, set $p\coloneqq p_0$ the unique source, ${\dsp E^{p}=\left\{pq_i\left|\right. 1\leq i\leq d\right\}=V_{c_1}}$ the $c_1$-vertices, and ${\dsp E_{c_1}\coloneqq\left\{Q_{ij}\right\}}$ the oriented $c_1$-edges where $Q_{ij}$ is the oriented $2$-face spanned by the oriented edges $i(Q_{ij})=pq_i$ and $t(Q_{ij})=pq_j$.  Note that the down connection on $\Gamma_{c_1}$ gives ${\dsp (\theta_c^d)_{Q_{ij}}(Q_{ik})=Q_{jk}}$ for all $1\leq i,j,k\leq d$.  Also by definition the compatibility system gives ${\dsp \left(\lambda_{c_1}^d\right)_{Q_{ij}}\left(Q_{ik}\right)=1}$ for all $1\leq i,j,k\leq d$.  Define constants ${\dsp \left\{m_{ij}\coloneqq\frac{\langle\xi,\alpha(pq_i)\rangle}{\langle\xi,\alpha(pq_j)\rangle}\left|\right. 1\leq i,j\leq d\right\}.}$  For all $1\leq i,j,k\leq d$ we have
\begin{align}
\label{eq:completeaxial1}
\alpha_{c_1}^d\left(Q_{ij}\right) & =-m_{ji}\cdot\alpha_{c_1}^d\left(Q_{ji}\right)\\
\label{eq:completeaxial2}
\alpha_{c_1}^d\left(Q_{ik}\right)-\alpha_{c_1}^d\left(Q_{jk}\right) & =m_{kj}\cdot\alpha_{c_1}^d\left(Q_{ij}\right).
\end{align}
Indeed the LHS of Eq. \eqref{eq:completeaxial2} gives
\begin{align*}
&\frac{\iota\left(\alpha(pq_i)\wedge\alpha(pq_k)\right)}{\langle\xi,\alpha(pq_i)\rangle}-\frac{\iota\left(\alpha(pq_j)\wedge\alpha(pq_k)\right)}{\langle\xi,\alpha(pq_j)\rangle}\\
=&\frac{\langle\xi,\alpha(pq_k)\rangle}{\langle\xi,\alpha(pq_j)\rangle}\alpha(pq_j)-\frac{\langle\xi,\alpha(pq_k)\rangle}{\langle\xi,\alpha(pq_i)\rangle}\alpha(pq_i)\\
=& m_{kj}\cdot\left(\alpha(pq_j)-\frac{\langle\xi,\alpha(pq_j)\rangle}{\langle\xi,\alpha(pq_i)\rangle}\alpha(pq_i)\right)
\end{align*}
which yields the RHS.  The verification of Eq. \eqref{eq:completeaxial1} is left to the reader.

Note that any function $A\colon E_{c_1}^{pq_1}\rightarrow\R^{d-1}$ that maps ${\dsp E_{c_1}^{pq_1}=\left\{Q_{12},Q_{13},\ldots, Q_{1d}\right\}}$ onto a basis extends to a generalized $d$-independent axial function $\hat{A}\colon E_{c_1}\rightarrow\R^{d-1}$ via the relations in Eqs. \eqref{eq:completeaxial1} and \eqref{eq:completeaxial2}.  Moreover $\alpha_{c_1}^d=p\circ\hat{A}$ where $p\colon\R^{d-1}\rightarrow\R^{n-1}$ is the surjective map mapping the basis vector $A\left(Q_{1i}\right)$ to the vector $\alpha\left(Q_{1i}\right)$.   

\end{proof}

\begin{figure}
\includegraphics[width=0.75\textwidth]{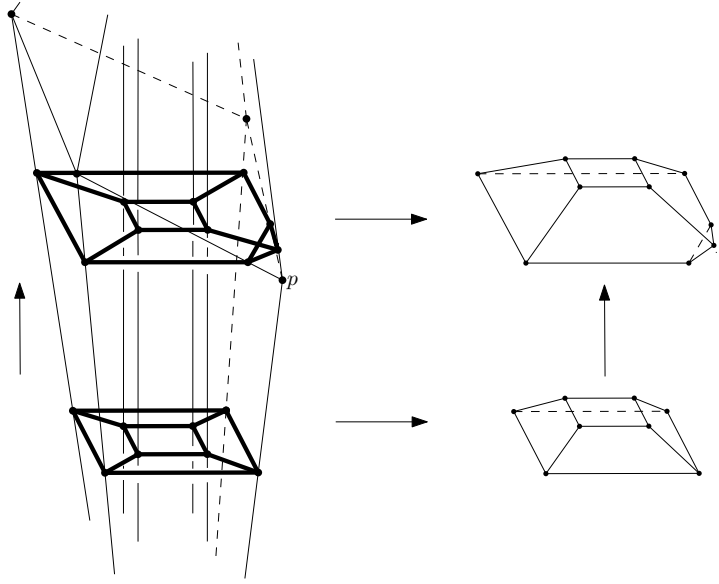}
\caption{A Commutative Diagram}
\label{fig:PrettyPassage}
\end{figure}

We are now in a position to explicitly write down a proof of Theorem \ref{thm:ch2mainresult}.  
\begin{proof}[Proof of Theorem \ref{thm:ch2mainresult}]
Assume $(\Gamma,A,\theta,\lambda)$ is a total lift of $(\Gamma,\alpha,\theta,\lambda)$, and let $Q$ be a $2$-face with normal edges $N_Q\subset E_\Gamma$.  Fix a base point $p_0$ in $Q$ and let $\gamma^Q\colon p_0\sra p_1\cdots\sra p_m\sra p_0$ denote the loop in $Q$ based at $p_0$.  Note that the $d$-independence of $A$ implies that $Q$ is actually a $2$-slice for some $2$-dimensional subspace $H\subset\R^d$.  Now for any edge $e\in N_Q^{p_0}$, the subspace spanned by $A(e)$ and $H$ must necessarily contain the vector $A\left(K^\perp_{\gamma^Q}\left(e\right)\right)$.  Hence by $d$-independence of $A$, we must have $K^\perp_{\gamma_Q}\left(e\right)=e$, hence $Q$ has trivial normal holonomy.  To compute the local normal holonomy number for an oriented edge $e'\in N_Q$, choose and fix a covector $\eta\in\left(\R^d\right)^*$ vanishing on the subspace $H$, but not vanishing on $A(e')$.  Then for $e\in Q$ at $i(e')$ we have
\begin{equation}
\label{eq:dindependentcompatiblity}
\alpha(e')-\frac{\eta\left(\alpha(e')\right)}{\eta\left(\alpha(\theta_e(e'))\right)}\alpha(\theta_e(e')) \ \in H \cap \spa\left\{A(e),A(e')\right\}= \spa\left\{A(e)\right\}.
\end{equation}
By $d$-independence of $A$, we deduce that ${\dsp\lambda_e(e')=\frac{\eta\left(\alpha(e')\right)}{\eta\left(\alpha(\theta_e(e'))\right)}}$, and thus readily conclude that ${\dsp \left|K_{\gamma^Q}^\perp\right|=1}$.  

Conversely, let $(\Gamma,\alpha,\theta,\lambda)$ be reducible $d$-valent 1-skeleton in $\R^n$ whose $2$-faces are all level with trivial normal holonomy.  Then by Lemma \ref{lem:cutlift}, there is a $(d+1)$-valent 1-skeleton $(\hat{\Gamma},\hat{\alpha},\hat{\theta},\hat{\lambda})$ in $\R^n\times\R$ that is reducible via the covector $\hat{\xi}=(\xi,\mathbf{1})\in\left(\R^n\times\R\right)^*$ and that also satisfies $(\dagger)$.  Moreover Lemma \ref{lem:cutlift} tells us that $(\Gamma,\alpha,\theta,\lambda)\equiv(\hat{\Gamma}_c,\pi\circ\hat{\alpha}_c^d,\hat{\theta}_c^d,\hat{\lambda}_c^d)$, for some linear isomorphism $\pi\colon W_{\hat{\xi}}\rightarrow\R^n$.  By Lemma \ref{lem:downcrosslift} $(\hat{\Gamma}_c,\hat{\alpha}_c^d,\hat{\theta}_c^d,\hat{\lambda}_c^d)$ must have a total lift, and hence so must $(\hat{\Gamma}_c,\pi\circ\hat{\alpha}_c^d,\hat{\theta}_c^d,\hat{\lambda}_c^d)$.  Thus $(\Gamma,\alpha,\theta,\lambda)$ also has a total lift.  

This establishes Theorem \ref{thm:ch2mainresult}.
\end{proof}

The proof of Corollary \ref{cor:simplepolytope}, which characterizes 1-skeleta of projected simple polytopes, appeals to a result of Guillemin and Zara dealing with the \emph{equivariant cohomology ring} of a 1-skeleton and its \emph{generating families}.  We refer the reader to \cite{GZ1} for definitions and further details regarding these concepts.
\begin{corollary1*}
A $d$-valent 1-skeleton $(\Gamma,\alpha,\theta,\lambda)$ is the 1-skeleton of a projected simple polytope if and only if $(\Gamma,\alpha,\theta,\lambda)$ is reducible, all its $2$-faces are level with trivial normal holonomy, and it admits an embedding $F\colon V_\Gamma\rightarrow \R^n$.
\end{corollary1*}  
\begin{proof}
Assume $(\Gamma,\alpha,\theta,\lambda)$ is the 1-skeleton of a projected simple $d$-polytope $S\subset\R^d$ via projection $p\colon\R^d\rightarrow\R^n$.  Then the 1-skeleton of $S$, $(\Gamma,A,\theta,\lambda)$, is a total lift of $(\Gamma,\alpha,\theta,\lambda)$ with respect to $p\colon\R^d\rightarrow\R^n$.  Thus by Theorem \ref{thm:ch2mainresult} all $2$-faces of $(\Gamma,\alpha,\theta,\lambda)$ are level with trivial normal holonomy.  Moreover the natural embedding 
$$\xymap{ V_S\ar[r]^-{F} & \R^d\\
v\ar@{|->}[r] & \vec{v}\\}$$
of $(\Gamma,A,\theta,\lambda)$ composes with the projection $p$ to give an embedding $f\coloneqq p\circ F\colon E_\Gamma\rightarrow \R^n$ of $(\Gamma,\alpha,\theta,\lambda)$.

Conversely assume that $(\Gamma,\alpha,\theta,\lambda)$ is a $d$-valent reducible 1-skeleton in $\R^n$ whose $2$-faces are level with trivial normal holonomy, and that admits an embedding $f\colon V_\Gamma\rightarrow\R^n$.  Then by Theorem \ref{thm:ch2mainresult}, $(\Gamma,\alpha,\theta,\lambda)$ has a total lift $(\Gamma,A,\theta,\lambda)\subset\R^d$ with respect to some projection $p\colon\R^d\rightarrow\R^n$.  Since $(\Gamma,A,\theta,\lambda)$ is a $d$-valent $d$-independent noncyclic 1-skeleton, Theorem 2.4.4 in \cite{GZ1} asserts that its equivariant cohomology ring admits a generating family $\left\{\tau_q\right\}_{q\in V_\Gamma}\subseteq H(\Gamma,A)$.  But the classes $\left\{p\circ\tau_q\right\}_{q\in V_\Gamma}\subseteq H(\Gamma,\alpha)$ are a generating family for the equivariant cohomology ring of the projection $(\Gamma,\alpha,\theta,\lambda)$.  This implies that composition with the projection $p\colon\R^d\rightarrow\R^n$ defines a surjective map on equivariant cohomology rings 
$$\xymap{H(\Gamma,A)\ar[r]^{p^*} & H(\Gamma,\alpha)\\ G\ar@{|->}[r] & p\circ G.\\}$$  
Since $f\in H^1(\Gamma,\alpha)$, there must exist $F\in H^1(\Gamma,A)$ satisfying $f=p\circ F$.  By linearity of $p$, $F\colon V_\Gamma\rightarrow \R^d$ must be an embedding of $(\Gamma,A,\theta,\lambda)$. 
Hence $(\Gamma,A,\theta,\lambda)$, a $d$-valent $d$-independent noncyclic 1-skeleton with an embedding, is the 1-skeleton of a simple $d$-polytope. 
\end{proof}

\section{Concluding Remarks}
Theorem \ref{thm:ch2mainresult} classifies projections of reducible (i.e. noncyclic) $d$-valent $d$-independent 1-skeleta using Morse theoretic ideas.  It is natural to ask if these same ideas could be applied toward classifying projections of reducible $d$-valent $k$-independent 1-skeleta for $k\geq 3$.  We observe that for $k\geq 3$ the $2$-faces of a $k$-independent 1-skeleton $(\Gamma,\alpha,\theta,\lambda)\subset\R^n$ are necessarily level.  Indeed fix a $2$-face $Q$, a basepoint vertex $p\in Q$, and edges $e_1,e_2\in E^p_Q$ that span $Q$.  Then for any normal edge to $Q$ at $p$, $e\in N^p_Q$, we have 
\begin{equation}
\label{eq:Conclusion1}
\alpha(e)-\lambda_{e_1}(e)\alpha(\theta_{e_1}(e))\in\R\cdot\alpha(e_1).
\end{equation}
Since $\alpha$ is $3$-independent, we can find a covector $\eta\in \left(\R^n\right)^*$ that vanishes on $\alpha(e_1),\alpha(e_2)$ but not on $\alpha(e)$.  Applying $\eta$ to both sides of \eqref{eq:Conclusion1} and solving for $\lambda$ we find that
\begin{equation}
\label{eq:Conclusion2}
\lambda_{e_1}(e)=\frac{\langle\eta,\alpha(e)\rangle}{\langle\eta,\alpha(\theta_{e_1}(e))\rangle}.
\end{equation}
Thus if $\gamma_Q$ is a loop around $Q$ satisfying $K_{\gamma_Q}^\perp(e)=e$, then Eq. \eqref{eq:Conclusion2} implies that the local normal holonomy number $\left|K_{\gamma_Q}^\perp\right|$ is equal to $1$.

In fact for $k\geq 4$, the $2$-faces of a $k$-independent 1-skeleton $(\Gamma,\alpha,\theta,\lambda)\subset\R^n$ must not only be level, but must also have trivial normal holonomy.  Indeed with $Q$, $p\in Q$, $e_1,e_2\in E^p_Q$, and $e\in N^p_Q$ as above, consider the four vectors $\alpha(e_1)$, $\alpha(e_2)$, $\alpha(e)$, and $\alpha\left(K_{\gamma_Q}^\perp(e)\right)$.  Note that these four vectors must all lie in the same $3$-dimensional subspace $\spa\left\{\alpha(e_1),\alpha(e_2),\alpha(e)\right\}$.  On the other hand, if $K_{\gamma_Q}^\perp(e)$ and $e$ are distinct, then the four vectors must be linearly independent, since $\alpha$ is $4$-independent.  Thus $K_{\gamma_Q}^\perp(e)=e$ and $Q$ has trivial normal holonomy.  Thus besides Corollary \ref{cor:simplepolytope}, Theorem \ref{thm:ch2mainresult} has the following, perhaps surprising, consequence:
\begin{corollary*}
Every $d$-valent $4$-independent noncyclic 1-skeleton has a total lift.
\end{corollary*} 
Therefore every 1-skeleton projection of a $4$-independent noncyclic 1-skeleton is in fact a projection of a $d$-independent noncyclic 1-skeleton.
In contrast, there are definitely $3$-independent noncyclic 1-skeleta that do not admit a total lift.  For example the 1-skeleton in Fig.~\ref{fig:NoLift} is $3$-independent noncyclic, but its triangular $2$-faces have non-trivial normal holonomy.

The following problem remains:
\begin{problem*}
Classify the 1-skeleton projections of $3$-independent noncyclic 1-skeleta.
\end{problem*}
 
If $(\Gamma,\alpha,\theta,\lambda)$ is a projection of a $3$-independent noncyclic 1-skeleton, it must be reducible and its $2$-faces must be level.  A first naive conjecture would be that these conditions are also sufficient, although a counterexample seems more likely.  

\begin{problem*}
Find an example of a reducible 1-skeleton with level $2$-faces that is not the projection of a noncyclic $3$-independent 1-skeleton.
\end{problem*}

As mentioned in the Remarks in Section 2:  It is unclear if every toral 1-skeleton is noncyclic.
\begin{problem*}
Find an example of a toral 1-skeleton that does not admit a polarizing covector.
\end{problem*}

Given a polarizing covector on a toral 1-skeleton in $\R^d$, and a compatible Morse function, the induced linear ordering of the vertices defines a shelling order on the $d$-cones of the corresponding fan.  As far as the author knows, showing that every complete simplicial fan has a shelling order is an open problem \cite[pg. 277]{Z}.  

Using the dictionary between 1-skeleta and GKM manifolds (e.g. \cite[Appendix B]{GZ2}), it would be interesting to know if there is an analogue of Theorem \ref{thm:ch2mainresult} in geometry.  For example, what are necessary and sufficient conditions for an effective Hamiltonian $T=\left(S^1\right)^n$-action on a symplectic $2d$-dimensional manifold to extend to a full $\left(S^1\right)^d$-action?




\bibliographystyle{plain}
\bibliography{thesis}   

\begin{thebibliography}{1}

\bibitem{CS}
Theodore Chang and Tor Skjelbred.
\newblock The topological {S}chur lemma and related results.
\newblock {\em Ann. of Math. (2)}, 100:307--321, 1974.

\bibitem{F}
William Fulton.
\newblock {\em Introduction to toric varieties}, volume 131 of {\em Annals of
  Mathematics Studies}.
\newblock Princeton University Press, Princeton, NJ, 1993.
\newblock The William H. Roever Lectures in Geometry.

\bibitem{GKM}
Mark Goresky, Robert Kottwitz, and Robert MacPherson.
\newblock Equivariant cohomology, {K}oszul duality, and the localization
  theorem.
\newblock {\em Invent. Math.}, 131(1):25--83, 1998.

\bibitem{GZ0}
V.~Guillemin and C.~Zara.
\newblock Equivariant de {R}ham theory and graphs.
\newblock {\em Asian J. Math.}, 3(1):49--76, 1999.
\newblock Sir Michael Atiyah: a great mathematician of the twentieth century.

\bibitem{GZ1}
V.~Guillemin and C.~Zara.
\newblock 1-skeleta, {B}etti numbers, and equivariant cohomology.
\newblock {\em Duke Math. J.}, 107(2):283--349, 2001.

\bibitem{GZ2}
Victor Guillemin and Catalin Zara.
\newblock The existence of generating families for the cohomology ring of a
  graph.
\newblock {\em Adv. Math.}, 174(1):115--153, 2003.

\bibitem{Z}
G{\"u}nter~M. Ziegler.
\newblock {\em Lectures on polytopes}, volume 152 of {\em Graduate Texts in
  Mathematics}.
\newblock Springer-Verlag, New York, 1995.

\end{thebibliography}

%
%

\end{document}